\documentclass[12pt,oneside]{amsart}
\usepackage{fullpage}
\usepackage{comment}
\usepackage{enumerate}
\usepackage{thmtools}
\usepackage{mathtools}
\usepackage[normalem]{ulem}
\usepackage{hyperref}
\usepackage{appendix}
\usepackage{amssymb}
\usepackage{amsthm}
\usepackage[normalem]{ulem}
\usepackage[mathscr]{euscript}
\usepackage{pgf,tikz}
\usetikzlibrary{arrows}
\usepackage{float}
\usepackage{hyperref}
\usepackage{cite}
\usepackage{color}
\usepackage{bbm}
\usepackage{xcolor}
\usepackage{transparent}
\usepackage{subfigure}
\usepackage{upgreek} % for $\upnu$, to distinguish $\nu$ from $v$.
\usepackage{amsfonts,txfonts}
\usepackage{bm}
\usepackage{graphicx}                       %Christian added this! idk what the dvips option is for so i just added without it
\usepackage{epstopdf}
\usepackage{psfrag}
\definecolor{azure(colorwheel)}{rgb}{0.0, 0.5, 1.0}
\definecolor{hanpurple}{rgb}{0.32, 0.09, 0.98}
\definecolor{iris}{rgb}{0.35, 0.31, 0.81}
\definecolor{byzantine}{rgb}{0.74, 0.2, 0.64}
\definecolor{ashgrey}{rgb}{0.7, 0.75, 0.71}
\definecolor{battleshipgrey}{rgb}{0.52, 0.52, 0.51}
\hypersetup{colorlinks=true,pdfborder=000,  citecolor  = blue, citebordercolor= {magenta}}
\hypersetup{linkcolor=black}

\makeatletter
\let\reftagform@=\tagform@
\def\tagform@#1{\maketag@@@{(\ignorespaces\textcolor{purple}{#1}\unskip\@@italiccorr)}}
\renewcommand{\eqref}[1]{\textup{\reftagform@{\ref{#1}}}}
\makeatother
\makeatletter
\DeclareUrlCommand\ULurl@@{%
  \def\UrlLeft{\uline\bgroup}%
  \def\UrlRight{\egroup}}
\def\ULurl@#1{\hyper@linkurl{\ULurl@@{#1}}{#1}}
\DeclareRobustCommand*\ULurl{\hyper@normalise\ULurl@}
\makeatother

\def\lessim{\ \lower4pt\hbox{$
		\buildrel{\displaystyle <}\over\sim$}\ }
\def\gessim{\ \lower4pt\hbox{$\buildrel{\displaystyle >}
		\over\sim$}\ }

\numberwithin{equation}{section}

\newcommand{\R}{\mathbb{R}}
\newcommand{\E}{\mathbb{E}}

\newcommand{\Prob}{\mathbb{P}}

\newcommand{\N}{\mathbb{N}}

\newcommand{\Z}{\mathbb{Z}}
\newcommand{\F}{\mathbb{F}}

\newcommand{\bnd}{\partial}

\newcommand{\ind}{\mathbbm{1}}

\newcommand{\diam}{\mathrm{diam}}
\newcommand{\Aut}{\mathrm{Aut}}
\newcommand{\connectsthru}[1]{\xleftrightarrow{#1}}
\newcommand{\connects}{\leftrightarrow}

\newtheorem{thm}{Theorem}
\newtheorem{lemma}{Lemma}[section]

\newtheorem{defn}{Definition}
\newtheorem{rmk}{Remark}[section]
\newtheorem{prop}[lemma]{Proposition}

\title{Chemical distance in graphs of polynomial growth}

	\author{Christian Gorski$^1$}
	\address[Christian Gorski]{Department of Mathematics, University of Washington, Seattle, USA}
	\email{cgorski1@uw.edu}
	\thanks{$^1$Department of Mathematics, University of Washington, Seattle, USA}

	\author{Eviatar B. Procaccia$^2$}
	\address[Eviatar B. Procaccia]{Faculty of Data and Decision Sciences, Technion - Israel Institute of Technology, Haifa, Israel}
	\email{eviatarp@technion.ac.il}
	\thanks{$^2$Faculty of Data and Decision Sciences, Technion - Israel Institute of Technology, Haifa, Israel}

\begin{document}
\begin{abstract}
We prove an Antal-Pisztora type theorem for transitive graphs of polynomial growth. That is,
we show that if $G$ is a transitive graph of polynomial growth and $p > p_c(G)$,
then for any two sites $x, y$ of $G$ which are connected by a $p$-open path,
the chemical distance from $x$ to $y$ is at most a constant times the original graph distance,
except with probability exponentially small in the distance from $x$ to $y$.
We also prove a similar theorem for general Cayley graphs of finitely presented groups,
for $p$ sufficiently close to 1.
Lastly, we show that all time constants for the chemical distance on the infinite supercritical
cluster of a transitive graph of polynomial growth are Lipschitz continuous as a function of $p$
away from $p_c$.
\end{abstract} 
\maketitle
\tableofcontents

\section{Introduction} \label{sec:intro}
Consider Bernoulli percolation on a graph $G = (V,E)$; we denote
by $G_p$ the associated subgraph of $G$.
For $x,y \in V(G_p)$ we define the intrinsic distance or ``chemical distance''
between $x$ and $y$ to be
\[
   d_{G_p}(x,y) := \inf \{ \ell(\pi) : \pi \mbox{ is an open path from } x \mbox{ to } y \mbox{ in } G_p \}.
\]
Here $\ell(\pi)$ is the length of $\pi$, that is, the number of edges in $\pi$, or equivalently one less
than the number of vertices in $\pi$.
Note that, for deterministic reasons, $d_G(x,y) \le d_{G_p}(x,y)$, 
where $d_G$ is the usual graph metric on $G$.
Antal and Pisztora \cite{AP} proved that, if $G$ is the standard Cayley graph
of $\Z^d$, in the supercritical phase, $d_{G_p}(x,y)$ is at most a constant times $d_G(x,y)$ with
very high probability as $d_G(x,y) \to \infty$.

We prove the following analogous theorem for transitive graphs of polynomial growth:

\begin{thm} \label{thm:main}
   Let $G = (V,E)$ be a transitive graph of polynomial growth.
   Then, for every $p_0 > p_c$, there exist $c(p_0) >0$ and $K(p_0) < \infty$
   such that for all $x,y \in G$, for all $t \ge d_G(x,y)$, and for all $p \in [p_0,1]$ we have
   \[
      \Prob \left( d_{G_p}(x,y) \ge K t , x \connectsthru{G_p} y \right) \le \exp(- c t).
   \]
   In particular, for some $C(p) < \infty$,
   \[
      \Prob \left( d_{G_p}(x,y) \ge K t \middle| x \connectsthru{G_p} y \right) \le C \exp(- c t).
   \]
\end{thm}

The same proof gives us a similar theorem for sufficiently large $p$ for any Cayley graph 
of a finitely presented group\footnote{In fact, the first part of Theorem \ref{thm:finitelypresented} holds
for any infinite connected graph which is $\Delta$-simply connected and has maximum degree at most $D$. We do not state that version of the theorem here
because (1) $\Delta$-simple connectedness has not been defined here yet; and (2) the second statement of the theorem
makes reference to the uniqueness threshold $p_u$ of $G$, which may not be defined if $G$ is not quasitransitive.}
\footnote{
Note also that the first part of Theorem \ref{thm:finitelypresented} is not interesting if $\Prob( x \connectsthru{G_p} y )$ already decays exponentially in $d_G(x,y)$.
Moreover, the second part of the theorem is vacuous unless $p_u <1$.
A theorem of Benjamini and Babson \cite{BenjaminiBabson} shows that when $G$ is the Cayley graph
of a finitely presented group, $p_u < 1$ if and only if $G$ is one-ended.}:
\begin{thm} \label{thm:finitelypresented}
   For each $D, \Delta \le \infty$, there exists $\tilde{p} = \tilde{p}(D, \Delta) < 1$ such that the following holds.
   Let $G = (V,E)$ be the Cayley graph of a finitely presented group $\langle S | R \rangle$
   with respect to the generating set $S$. Suppose that $|S| \le D$ and
   suppose that every relator in $R$ has word length at most $\Delta$.
   Then, for every $p_0 > \tilde{p}$, there exist $c(p_0) >0$ and $K(p_0) < \infty$
   such that for all $x,y \in G$, for all $t \ge d_G(x,y)$, and for all $p \in [p_0,1]$ we have
   \[
      \Prob \left( d_{G_p}(x,y) \ge K t , x \connectsthru{G_p} y \right) \le \exp(- c t).
   \]
   In particular, if in addition $p > p_u$ (where $p_u$ is the uniqueness threshold for $G$),
   then for some $C(p) < \infty$,
   \[
      \Prob \left( d_{G_p}(x,y) \ge K t \middle| x \connectsthru{G_p} y \right) \le C \exp(- c t).
   \]
\end{thm}
The proof of Theorem \ref{thm:finitelypresented} is presented at the end of Section \ref{sec:proof_of_lem_problem}.
Although it seems plausible that an Antal-Pisztora-like theorem could hold
in the uniqueness regime of any (infinite bounded degree) quasitransitive graph,
to our knowledge, until now it has not been shown even for sufficiently large $p$
in any geometry which is not ``$\Z^d$-like''. 

Theorem \ref{thm:main} also allows us to prove Lipschitz continuity of the time constants associated to the chemical distance
on the infinite supercritical percolation cluster for transitive graphs of polynomial growth 
(see the discussion at the beginning of Section \ref{sec:continuity}).
That is, defining $D_p(o,x)$ to be the chemical distance between the sites of the infinite cluster
which are $d_G$-closest to $o$ and $x$
respectively, we have the following theorem:
\begin{thm} \label{thm:cty}
   Above criticality, the time constants associated to $D_p$ are Lipschitz continuous in $p$ in the following sense.
   For each $p_0 > p_c$, there exists $C(p_0)$ such that
   \[
      \limsup_{x \to \infty} \left| \frac{ \E D_p(o, x) }{ d_G(o,x) } - \frac{\E D_q(o,x) }{ d_G(o,x) } \right|
      \le C(p_0)|p-q|
   \]
   for all $p,q \in [p_0, 1]$.
   In particular, given any sequence $x_n \to \infty$ in $V$, if we define
   \begin{align*}
      \overline{d}(p) := \limsup_{n \to \infty} \frac{ \E D_p(o,x_n) }{ d_G(o, x_n) } & &
      \underline{d}(p) := \liminf_{n \to \infty} \frac{\E D_p(o,x_n) }{ d_G(o,x_n) },
   \end{align*}
   then for any $p_0 > p_c$, both $\overline{d}$ and $\underline{d}$ 
   are Lipschitz functions of $p$ on $[p_0,1]$.
\end{thm}
If (as is expected, but has not yet been shown in the literature) there exists a scaling limit for the supercritical percolation cluster in this case,
then Theorem \ref{thm:cty} will show that the scaling limit is a continuous function of $p$.\footnote{To make sense of this statement,
we need a topology on the set of scaling limits. Here, all the scaling limits will be given by
metrics on a particular Lie group (diffeomorphic to $\R^n$). Therefore we can give the space of metrics the topology of 
uniform convergence on compact sets, and then Theorem \ref{thm:cty} establishes that the scaling limits are a continuous 
function of $p$. We can also consider the scaling limits as abstract metric spaces, and give them the 
topology of pointed Gromov-Hausdorff convergence, and we also get continuity with respect to this topology.}

The proof of Theorem \ref{thm:main} has the same basic outline as the proof of Antal and Pisztora. 
Our main contribution is a more elegant geometric construction (with fewer hypotheses) of an open path from $x$ to $y$
which is close to an edge-geodesic.
%only requires uniform control on the degree and degree of coarse simple-connectedness. 
This construction allows us to avoid the tedious case-checking which occurs in e.g. Lemma 3.2 of \cite{AP}.
The construction is also \emph{necessary} for tackling general transitive graphs of polynomial growth;
there is no general procedure for creating a \emph{transitive} coarse-graining of such a graph,
and so any geometric argument which is to be applied to a coarse-graining must require rather weak 
geometric assumptions.
In addition, the geometric construction gives rise almost immediately
to Theorem \ref{thm:finitelypresented}, a theorem which applies to a huge class of graphs
and (to our knowledge) is new for any graph which is not quasi-isometric to $\Z^d$.

The basic outline of the proof of Theorem \ref{thm:cty} is similar to that of \cite{CNN};
some subtle differences between our approach and that of \cite{CNN} are exposited at the end of the
first subsection of Section \ref{sec:continuity}.
The observation that the lattice animal bounds of Lemma \ref{lem:linear-bound} hold
for general graphs of polynomial growth is to our knowledge new, and relies on a basic
geometric property of our coarse-graining construction which did not appear in \cite{Gorski},
where the coarse-graining construction was introduced.

%\textcolor{red}{Our main contribution is in the geometric control of the coarse graining.
%Framing in terms of coarse simple connectedness--more elegant geometric argument, e.g. don't have to deal explicitly with nesting--theorem 2 also new.}
\section{Proof of Theorem \ref{thm:main}}
\subsection{Outline of proof of Theorem \ref{thm:main}}
The proof of Theorem \ref{thm:main} proceeds as follows.
Quantitative uniqueness results of Contreras, Martineau, and Tassion \cite{CMT} tell us
that at large scales, the overwhelming majority of regions of the graph 
have a unique giant component (and adjacent giant components ``glue together''). Therefore, fixing a sufficiently large scale,
one should expect to be able to connect $x$ to $y$ almost directly through these ``good regions.''
To make this precise, we need to describe a procedure for coarse-graining the graph 
(following \cite{Gorski}),
show that we have uniform control (over arbitrarily large scales) of a geometric property called \emph{coarse simple connectedness},
and then use this geometric control to argue that we can indeed connect $x$ and $y$ almost directly when obstacles 
(the ``bad regions'' of the coarse graining) are sparse enough.
These latter two points are our main contributions.
In the rest of this section, we will describe the coarse graining, state some lemmata, and then prove Theorem \ref{thm:main}
modulo these lemmata.

First, let us describe the coarse-graining procedure, which is essentially the same as that used in \cite{Gorski}. For a fixed sufficiently large scale $R$, take $V^R$ to be a maximal $(R/30)$-separated
subset of $V$. (Recall that a subset $S \subset V$ is \emph{$r$-separated} if for all $x \ne y \in V$ we have $d_G(x,y) \ge r$.)
Fix a ``closest-point projection map'' $\rho: V \to V^R$ such for each $x$, $\rho(x) \in \{v \in V^R : \forall w \in V^R, d_G(v,x) \le d_G(w,x) \}$. That is, $\rho$ sends each vertex of $G$ to the nearest point in $V^R$, but can ``break ties'' arbitrarily.
For each $v \in V^R$ we define the Voronoi tile $\hat{v}$ to be $\hat{v} := \rho^{-1}(v) \subset V$, the set of vertices in $V$
closest to $v$ with ties broken by $\rho$. 
We can choose $\rho$ such that each Voronoi tile $\hat{v}$ is star-shaped about $v$, that is, if $x \in \hat{v}$,
then every edge-geodesic from $x$ to $v$ is contained in $\hat{v}$. This will happen automatically,
for instance, if ties are broken according to some total ordering on $V^R$.
We define the coarse-grained graph $\hat{G} = (\hat{V}, \hat{E})$ of $G$ by
\[
   \hat{V} := \{ \hat{v} : v \in V^R \}
\]
\[
   \hat{E} := \{ \{\hat{v}, \hat{w} \} : \exists x \in \hat{v}, y \in \hat{w} \mbox{ such that } \{x,y\} \in E \}.
\]
We define a map $\hat{\rho}: G \to \hat{G}$ from $G$ to its coarse-graining by $\hat{\rho}(v) = \hat{\rho(v)}$, i.e. each vertex of $G$
is sent to the unique Voronoi tile it belongs to (which is itself a vertex of $\hat{G}$).
Note that $\hat{\rho}$ is a distance decreasing map: each pair of adjacent points in $G$ is either sent
to a pair of adjacent points in $\hat{G}$, or both are sent to the same point of $\hat{G}$.
%\textcolor{red}{TODO: prove a few properties, eg bounded degree of $\hat{G}$, bi-Lipschitz.
%More precisely, $\exists 0 < k < K < \infty$ such that for all $\hat{v},\hat{w} \in \hat{V}$ we have
%$(k/R) d_G(v,w) \le d_{\hat{G}}(\hat{v},\hat{w}) \le (K/R) d_G(v,w)$.}

A key geometric property of the coarse-graining is that of \emph{coarse simple connectedness}:
\begin{defn}
   Let $0 \le \Delta < \infty$. A graph $G = (V,E)$ is called \emph{$\Delta$-simply connected}
   if for any edge cycle $C \subset E$, there exist edge cycles $C_1,...,C_N$ with each $\diam C_i \le \Delta$
   such that $C = C_1 \oplus \cdots \oplus C_N$,
   where $A \oplus B$ denotes the symmetric difference of sets $(A \setminus B) \sqcup (B \setminus A)$.
   That is, every cycle of $G$ is generated (under symmetric difference) by cycles of diameter at most $\Delta$.
   If $G$ is $\Delta$-simply connected for some $\Delta < \infty$, then we call $G$ \emph{coarsely simply connected}.
\end{defn}
See Figure \ref{fig:obstacles} for an illustration of coarse simple connectedness.
As a simple example, the standard Cayley graph of $\Z^d$ 
(i.e. the $d$-dimensional hypercubic lattice)
is $2$-simply connected, because every cycle is a symmetric difference of squares (of diameter $2$).

Note that the coarse-graining we constructed is typically not transitive and moreover
depends on the scale $R$.
However, the following proposition, proved in Section \ref{sec:coarse}, gives us \emph{uniform} control
over the geometry of the coarse-graining $\hat{G}$, independent of $R$.
\begin{prop} \label{prop:uniformoverscales}
   Let $G$ be a transitive graph of polynomial growth. Then there exist $\hat{D} < \infty$ and $\Delta < \infty$ such
   that for \emph{any} scale $R$, the associated coarse-graining $\hat{G}$ has maximum degree at most $\hat{D}$
   and is $\Delta$-simply connected.
\end{prop}

Now, given a percolation configuration on $G$, we associate to it a site percolation configuration on $\hat{G}$ as follows.
For each $v \in V^R$, define the event
\begin{equation} \label{eq:goodevent}
   A_v := 
   \left \{ B(v, \frac{R}{10}) \connectsthru{G_p} \bnd B(v,R) \right\} \cap
   \left\{ \begin{array}{c}
       \mbox{ at most one component of } G_p \cap B(v,R) \\
   \mbox{ intersects both } B(v, \frac{R}{5}) \mbox{ and } \bnd B(v, \frac{R}{2}) 
   \end{array} \right\}.
\end{equation}
We say that a site $\hat{v} \in \hat{V}$ is \emph{open} if $A_v$ holds and \emph{closed} otherwise.

By Proposition \ref{prop:separate} below,
there exists some $K' < \infty$ (independent of $R$) such that if $d_{\hat{G}}(\hat{v}, \hat{w}) \ge K'$, then $B_G(v,R)$ and $B_G(w,R)$
are disjoint and hence
$A_v$ and $A_w$ are independent. Moreover, by \cite{CMT} we have
the following. 
\begin{prop}[Proposition 1.3 of \cite{CMT}] \label{prop:quantunique}
   Let $p_0 > p_c$. Then there exists $R_0 < \infty$ such that for all $R \ge R_0$, we have
   \[
      \Prob_p( A_v) \ge 1 - e^{- \sqrt{R}}
   \]
   for all $p \in [p_0,1]$.
\end{prop}
\begin{rmk}
   The wording in \cite{CMT} may leave ambiguous whether one $R_0$ will work for \emph{all} $p \in [p_0,1]$
   simultaneously; an analysis of the proof of Proposition 3.1 of \cite{CMT} reveals that indeed the choice of $R_0$
   does not depend on $p$, only on $p_0$. (In the notation of \cite{CMT}, the required lower bound
   on $n$ only depends on $p$, not on $q$.)
\end{rmk}
Thus our site percolation on $\hat{G}$ has bounded dependence and retention probability tending to 1,
allowing us to use the method of domination by product measures \cite{LSS}.

Now, take $x, y \in V$. Define the \emph{closed $\Delta$-cluster} of a point $\hat{v} \in \hat{G}$
to be empty if $\hat{v}$ is open and otherwise equal to the connected component of $\hat{v}$
in the graph whose vertices are the closed sites in $\hat{G}$ and whose edges
connect every pair of closed sites in $\hat{G}$ which have $d_{\hat{G}}$-distance
at most $\Delta$.
Now we take $F$ to be the union of the closed clusters of all the sites of $\hat{G}$
intersecting 
\[
   I_{x,y} := B_{\hat{G}}(\hat{\rho}(x), K') \cup B_{\hat{G}}(\hat{\rho}(y), K') \cup [\hat{\rho}(x), \hat{\rho}(y)].
\]
Again $K'$ is as given in Proposition \ref{prop:separate} below, and $[\hat{\rho}(x), \hat{\rho}(y)]$
is a shortest path from $\hat{\rho}(x)$ to $\hat{\rho}(y)$ in $\hat{G}$ chosen by some deterministic rule. 
%(For every graph $H$ appearing in this paper, we fix in advance 
%a particular choice of edge geodesic $[u,v]$ between each pair of vertices $u,v \in H$.)

Throughout the paper, we will denote by $N(S,r)$ the $r$-neighborhood of the set $S$, that is, $N(S,r) = \bigcup_{s \in S} B(s,r)$.
To prove Theorem \ref{thm:main}, the key geometric argument will be 
\begin{lemma} \label{lem:geolem}
   Suppose that $x \connectsthru{G_p} y$. Then there exists an open path in $G$ from $x$ to $y$ which lies in
   \[
      \bigcup \{ B(v,R) : \hat{v} \in I_{x,y} \cup N(F, \Delta) \}. 
   \]
\end{lemma}

The key probabilistic argument will be
\begin{lemma} \label{lem:problem}
   Given $p > p_c$, for all $R$ sufficiently large, there exists $K$ such that for all $x, y \in V$, 
   there exists $C, c > 0$ such that for all $t \ge d_G(x,y)$ we have
   \[
      \Prob_p( |F| > K t ) \le C \exp( -c t ).
   \]
\end{lemma}

Given these two lemmata, the theorem quickly follows:
\begin{proof}[Proof of Theorem \ref{thm:main} given Lemmas \ref{lem:geolem} and \ref{lem:problem}]
   Let $p \ge p_0 > p_c$. Note that since $G$ is of polynomial growth, it is amenable, and so
   by the Burton-Keane theorem \cite{burton1989density, haggstrom2006uniqueness} there is exactly one infinite connected component of
   $G_p$, and so by the Harris-FKG inequality \cite[Theorem 2.4]{grimmett1999percolation}
   $\Prob_p(x \connects y) \ge \Prob_p(x \connects \infty) \Prob_p(y \connects \infty) = \Prob_p(o \connects \infty)^2 > 0$.
   That is, the connectivity probability is bounded below by a constant independent of $x$ and $y$.
   Thus, to prove that $\Prob_p( d_{G_p}(x,y) \ge K t | x \connectsthru{G_p} y )$ decays
   exponentially in $t$ (for $t \ge d_G(x,y)$), it suffices to show that $\Prob_p( d_{G_p}(x,y) \ge K t , x \connectsthru{G_p} y )$
   decays exponentially in $t$.
   
   So using Lemma \ref{lem:problem}, fix $R < \infty$ and $\kappa$ sufficiently large such that $\Prob_p( |F| > \kappa t )$ decays exponentially 
   in $t$ for $t \ge d_G(x,y)$. 
   If $D$ is the maximum degree of $G$ and $\hat{D}$ is the maximum degree of $\hat{G}$,
   then on the event $\{x \connectsthru{G_p} y\}$, Lemma \ref{lem:geolem} gives us the very crude bound
   \[
      d_{G_p}(x,y) \le M( 2\hat{D}^{K'} + d_G(x,y) + \hat{D}^{\Delta} |F| )
   \]
   where $M := |B_G(R)| \le D^R$.
   In particular, on the (exponentially likely in $t$) event that $|F| \le \kappa t$, this gives that
   \[
      d_{G_p}(x,y) \le M( 2\hat{D}^{K'} + d_G(x,y) + \hat{D}^{\Delta} \kappa t ) \le K t
   \]
   where the last inequality will hold whenever $d_G(x,y) \ge 2\hat{D}^{K'}$ if we take
   $K = M(2 + \hat{D}^{\Delta} \kappa)$. That is,
   \[
      \Prob_p( x \connectsthru{G_p} y, d_{G_p}(x,y) > K t ) \le 
      \Prob_p( |F| > K t ) \le C \exp(-c t )
   \]
   for all $t \ge d_G(x,y)$ whenever $d_G(x,y)$ is sufficiently large; and so we are done.
\end{proof}

\subsection{Geometric properties of the coarse-graining $\hat{G}$} \label{sec:coarse}
Here we prove some geometric properties of the coarse-graining $\hat{G}$ which are needed for the
proofs of Lemmata \ref{lem:geolem} and \ref{lem:problem}.
We begin with Proposition \ref{prop:uniformoverscales},
which controls the maximum degree and the coarse simple connectedness parameter for $\hat{G}$
uniformly as $R \to \infty$.

%Being coarsely simply connected is a quasi-isometry invariant of graphs (see e.g. \cite{GGTDrutuKapovich} Corollary 9.36).
%It is well-known in geometric group theory (see e.g. Corollary 9.5 of \cite{GGTDrutuKapovich}) that
%a Cayley graph of a group is coarsely simply connected if and only if 
%the underlying group is finitely presented.
%In particular, since any nilpotent group is finitely presented (see e.g. Proposition 13.75 and 13.84 in \cite{GGTDrutuKapovich}),
%and by Trofimov's theorem \cite{Trofimov} any transitive graph of polynomial growth
%is quasi-isometric to a Cayley graph of a nilpotent group,
%any transitive graph of polynomial growth is coarsely simply connected.
%For our purposes, we will actually use a coarse-graining procedure
%on a polynomial growth transitive graph (which produces
%a graph which is not necessarily transitive), and
%we will need to check that if the original graph
%is $\Delta$-simply connected,
%the coarse-grained graph is also $\Delta$-simply connected.

\begin{proof}[Proof of Proposition \ref{prop:uniformoverscales}]
   First we bound the maximum degree of $\hat{G}$.
   Maximum degree bounds were already proven in \cite{Gorski} but we reiterate
   the argument here. First, note that because $V^R$ is $R/30$-separated, for every $\hat{v} \in \hat{V}$
   we have
   \[
      B_G(v, R/60 - 1) \subset \hat{v} \subset B_G(v, R/30).
   \]
   If $\hat{v}$ is adjacent to $\hat{w}$, the above shows that there is an edge
   connecting $B_G(v,R/30)$ to $B_G(w,R/30)$, so we then have that $d_G(v,w) \le (R/15) + 1$.
   Therefore we have the inclusion of disjoint unions
   \[
      \bigsqcup_{\hat{w} :  d_{\hat{G}}(\hat{v}, \hat{w}) \le 1} B_G(w, R/60 - 1) \subset
      \left(\bigsqcup_{\hat{w} :  d_{\hat{G}}(\hat{v}, \hat{w}) \le 1} \hat{w} \right)
      \subset B_G(v, R/10),
   \]
   so comparing cardinalities (and using transitivity) gives the bound
   \[
      \mathrm{deg} \hat{v} + 1 \le \frac{ |B_G(o, R/10)| }{ |B_G(o, R/60 - 1) }.
   \]
   Since $G$ is transitive of polynomial growth, by Trofimov's theorem \cite{Trofimov}
   it is quasi-isometric to a Cayley graph of a finitely generated nilpotent group,
   and therefore the function $R \mapsto |B_G(o,R)|$ is bounded above and below
   by polynomials in $R$ of the same degree. Therefore the fraction above is bounded
   by a constant $\hat{D}$ independent of $R$, and we have proved the desired degree bound.
   
   Now for coarse simple connectedness.
   Since finitely generated nilpotent groups are finitely presented (see e.g. Proposition 13.75 and 13.84 in \cite{GGTDrutuKapovich}), they are coarsely simply connected
   (see e.g. Corollary 9.5 of \cite{GGTDrutuKapovich}).
   Since coarse simple connectedness is a quasi-isometry invariant of graphs (see e.g. \cite{GGTDrutuKapovich} Corollary 9.36),
   again by Trofimov's theorem \cite{Trofimov}
   $G$ is $\Delta$-simply connected for some $\Delta < \infty$. 
   We now show that each coarse-graining $\hat{G}$ at every scale $R$ is \emph{also} $\Delta$-simply connected.\footnote{
   Each coarse-graining $\hat{G}$ is quasi-isometric to $G$, and hence
   we know that it is $\Delta'$-simply connected for \emph{some} $\Delta'$ which a priori depends on $R$.
   However, for our purposes we need uniform control of the parameter $\Delta'$ as $R$ increases,
   so we cannot just appeal to quasi-isometry.}
   
   To do this, it will be helpful to use the language of homology.
   Recall that a \emph{1-chain in } $G$ \emph{ with coefficients in } $\F_2$ is a finite formal linear combination
   of edges of $G$ with coefficients in $\F_2 = \Z / 2\Z = \{ \bar{0}, \bar{1} \}$.
   The collection of all $1$-chains in $G$ is denoted by $C_1(G, \F_2)$, and this can 
   naturally be identified with the collection of all finite subsets of $E$.
   Under this identification, the natural addition operation on $C_1(G, \F_2)$ corresponds
   to symmetric difference $\oplus$ of sets.
   We can similarly define $C_0(G, \F_2)$ to be the collection of all finite linear combinations of \emph{sites}
   of $G$ (or equivalently, all finite subsets of $V$).
   We can then define the boundary operator $\partial: C_1(G, \F_2) \to C_0(G, \F_2)$
   to be the unique $\F_2$-linear map that sends each edge to its endpoints.
   Note that this corresponds to the usual notion of boundary;
   for instance, if $\pi \in C_1(G, \F_2)$ is a path of edges from $x$ to $y$, then due to cancellation, $\partial \pi = x \oplus y$.
   Crucially, each $c \in C_1(G, \F_2)$ satisfies $\partial c = 0$ if and only
   if $c$ is a (possibly empty) finite union of cycles in the usual graph-theoretic sense.
   
   Now consider a cycle $\hat{q}$ in $\hat{G}$. We want to first find a ``corresponding'' cycle $q$ in $G$.
   To see what we mean by this, note that the map $\hat{\rho}:V \to \hat{V}$ defined by $\hat{\rho}(x) = \hat{\rho(x)}$
   (where $\rho:V \to V^R$ is the nearest point projection we used to define the coarse-graining)
   has the property that if $\{v,w\} \in E$, either $\hat{\rho}(v) = \hat{\rho}(w)$, or $\{ \hat{\rho}(v), \hat{\rho}(w) \} \in \hat{E}$.
   Thus, $\hat{\rho}$ induces a $\F_2$-linear map of 1-chains $\hat{\rho} : C_1(G, \F_2) \to C_1(\hat{G}, \F_2)$,
   where we take $\hat{\rho}(\{v,w\}) := 0$ if $\hat{\rho}(v) = \hat{\rho}(w)$ and
   $\hat{\rho}(\{v,w\}) = \{\hat{\rho}(v), \hat{\rho}(w) \}$ otherwise.
   Note that this is a chain map, that is, it preserves the boundary operator in the sense
   that $\bnd \hat{\rho}(c) = \hat{\rho} (\bnd(c))$.
   
   Now, given a cycle $\hat{q}$ in $\hat{G}$, we want to find a cycle $q$ in $G$ which is ``corresponding''
   in the sense that $\hat{\rho}(c) = \hat{c}$.
   To do this, first consider the chain $\hat{q}$ as a sequence of vertices $\hat{v}_0 ,\hat{v}_1,..., \hat{v}_N = \hat{v}_0$
   in $\hat{V}$ with the property that each $\{ \hat{v}_i, \hat{v}_{i+1} \} \in \hat{E}$.
   By definition of $\hat{E}$, for each $i = 0,..., N-1$ there exists $e_i = \{b_i, a_{i+1} \} \in E$
   such that $b_i \in \hat{v}_i$ and $a_{i+1} \in \hat{v}_{i+1}$. Then we can take $q$ to be the concatenation of edge paths
   \[
      q = [v_0, b_0] \oplus e_0 \oplus [a_1, v_1] \oplus [v_1, b_1] \oplus \cdots \oplus e_{N-1} \oplus [a_N, v_N],
   \]
   where each $[\xi, \eta]$ is some edge geodesic from $\xi$ to $\eta$ in $G$.
   By star-convexity of all the Voronoi tiles, each $[a_i, v_i] \oplus [v_i, b_i]$ is contained entirely in $\hat{v}$,
   meaning that $\hat{\rho}([a_i, v_i] \oplus [v_i, b_i]) = 0$. So we see that
   \[
      \hat{\rho}(q) = \hat{\rho}(e_0) \oplus \hat{\rho}(e_1) \oplus \cdots \oplus \hat{\rho}(e_{N-1}) = \hat{q},
   \]
   as desired.
   $q$ is a cycle in $G$, and so since $G$ is $\Delta$-simply connected, there exists a family 
   $\mathcal{C}$ of cycles of diameter at most $\Delta$ such that $q = \bigoplus_{ c \in \mathcal{C}} c$.
   But then we have (by linearity)
   \[
      \hat{q} = \hat{\rho}(q) = \bigoplus_{c \in \mathcal{C}} \hat{\rho}(c),
   \]
   so it only remains to show that each $\hat{\rho}(c)$ is either $0$ or a cycle in $\hat{G}$ of diameter at most $\Delta$.
   But since $\hat{\rho}$ is a chain map, for each cycle $c \in \mathcal{C}$, $\bnd \hat{\rho}(c) = \hat{\rho}( \bnd c ) = 0$,
   and so $\hat{\rho}(c)$ is indeed a cycle. To see the diameter bound, simply note that $\hat{\rho}: V \to \hat{V}$
   is a distance-decreasing map.
\end{proof}

We also will need the following proposition:
\begin{prop} \label{prop:separate}
   There exists $K' < \infty$ independent of $R$ such that for any sufficiently large scale $R$, we have that
   for all $\hat{v}, \hat{w} \in \hat{G}$,
   \[
      d_{\hat{G}}(\hat{v}, \hat{w}) \ge K' \Rightarrow B_G(v,R) \cap B_G(w,R) = \emptyset.
   \]
\end{prop}
\begin{proof}
   $B_G(v,R) \cap B_G(w,R) = \emptyset$ is equivalent to $d_G(v,w) \ge 2R + 1$.
   Therefore to show the proposition it will suffice to show that there exists $C < \infty$ such that
   $d_{\hat{G}}(\hat{v}, \hat{w}) \le \frac{C}{R} d_G(v,w)$ for all $\hat{v} \ne \hat{w} \in \hat{G}$, 
   since then we can just take $K' := 3C$.
   
   To see this bound, let $\hat{v}, \hat{w} \in \hat{G}$ and consider a geodesic $\pi$
   in $G$ from $v$ to $w$. $\pi$ can be covered by at most $(10d_G(v,w)/R + 1)$
   balls of radius $R/10$, and each of these balls contains at most $\hat{D}$
   distinct tiles $\hat{x}$, by the same argument as in the proof of Proposition \ref{prop:uniformoverscales}.
   Therefore $\hat{\rho}(\pi)$ is a path in $\hat{G}$  consisting of at most
   $[(10 d_G(v,w)/R) + 1]\hat{D}$ sites, and so (using the fact that $\hat{v} \ne \hat{w}$ implies that $d_G(v,w) \ge R/30$
   and hence $1 \le 30 d_G(v,w)/R$)
   we have
   \[
      d_{\hat{G}}(\hat{v}, \hat{w}) \le \frac{40 \hat{D}}{R} d_G(v,w),
   \]
   which is the desired bound.
\end{proof}

\subsection{Proof of Lemma \ref{lem:geolem}}

In this section we prove Lemma \ref{lem:geolem}.
First, let us establish some facts about the relationship between the ``microscopic'' percolation process $G_p$
and the induced ``macroscopic'' percolation process on $\hat{G}$. Recall that 
we define a vertex $\hat{v} \in \hat{G}$ to be open if the event $A_v$ defined in \eqref{eq:goodevent} holds,
and closed otherwise.

First, note that the event $A_v$ implies that there is exactly one connected component of $G_p \cap B(v,R)$
intersecting both $B(v,R/5)$ and $\bnd B(v, R/2)$. We refer to this component as ``the giant component near $v$''
and denote it by $\kappa(v)$.
\begin{prop} \label{prop:macrotomicro}
   Suppose that $\hat{v}$ and $\hat{w}$ are endpoints of an open path $\hat{\gamma}$
   in $\hat{G}$. Then for any
   $\xi \in \kappa(v), \zeta \in \kappa(w)$, there exists an open path in $G_p \cap \bigcup_{\hat{u} \in \hat{\gamma}} B(u,R)$
   which connects $\xi$ to $\zeta$.
\end{prop}
\begin{proof}
   First, note that since $\kappa(v)$ is a connected component of $G_p \cap B(v,R)$, every pair of points
   in $\kappa(v)$ is joined by an open path in $G_p \cap B(v,R)$, so the case that $\hat{v} = \hat{w}$ is immediate.
   Now suppose that we have proved the proposition for all path $\hat{\gamma}$ of length at most $k$
   for some $k \ge 1$. Consider a path $\hat{\gamma}$ of length $k + 1$ from $\hat{v}$ to $\hat{w}$
   and let $\xi \in \kappa(v)$, $\zeta \in \kappa(w)$.
   Let $\hat{\gamma}'$ be the initial subpath of $\hat{\gamma}$ of length $k$, and denote by $\hat{w}'$ its
   ending point, and note that $\hat{w'}$ is adjacent to $\hat{w}$ in $\hat{G}$.
   Since $A_{w'}$ holds, we have that there exists an open path in $G_p$ from $\bnd B(w', R)$
   to some point $\zeta' \in B(w', R/10)$; this path is then contained in $\kappa(w')$, since it necessarily intersects
   both $\bnd B(w', R/2)$ and $B(w', R/5)$. On the other hand, since $\hat{w}'$ and $\hat{w}$
   are adjacent in $\hat{G}$, we have that $d_G(w', w) \le R/10$, so by the triangle inequality
   the same open path begins outside $B(w, (9R)/10)$ and ends inside $B(w, R/5)$.
   Thus, since $A_w$ holds, the intersection of this path with $B(w,R)$ (which contains $\zeta'$) lies in $\kappa(w)$.
   Now, since $\zeta' \in \kappa(w')$, by the inductive hypothesis, there exists an open path from $\xi$ to $\zeta'$
   which is contained in $G_p \cap \bigcup_{\hat{x} \in \hat{\gamma}'} B(x,R)$ from $\xi$ to $\zeta'$.
   On the other hand, since $\zeta' \in \kappa(w)$, there is an open path in $G_p \cap B(w,R)$ from $\zeta'$ to $\zeta$.
   The concatenation of these open paths joins $\xi$ to $\zeta$ and lies in
   $G_p \cap \left( \bigcup_{\hat{x} \in \hat{\gamma}'} B(x,R) \cup B(w,R) \right) = 
   G_p \cap \bigcup_{\hat{x} \in \hat{\gamma}} B(x,R)$, and so we are done.
\end{proof}

The next proposition tells us that we may glue the microscopic open path promised by Proposition \ref{prop:macrotomicro}
to an existing open microscopic path, so long as the attempted gluing does not occur too close to the endpoints.
\begin{prop} \label{prop:surgery}
   Let $x,y \in V$ and let $\pi$ be an open path from $x$ to $y$ in $G_p$. Let $\hat{\pi} := \hat{\rho}(\pi)$
   be its image under the map $\hat{\rho}: G \to \hat{G}$ (which is not
   necessarily an open path in $\hat{G}$). Suppose that there exists a subpath 
   $\hat{\gamma}$ of $\hat{\pi}$ which has open endpoints, and suppose that $\hat{\gamma}'$
   is an open path in $\hat{G}$ which has the same endpoints as $\hat{\gamma}$. 
   Further suppose that the endpoints of $\hat{\gamma}$ lie at $d_{\hat{G}}$-distance
   at least $K'$ from both $\hat{\rho}(x)$ and $\hat{\rho}(y)$.
   %Let $\hat{\pi}'$ be the path in $\hat{G}$ which is obtained from $\hat{\pi}$ by replacing
   %the subpath $\hat{\gamma}$ by $\hat{\gamma}'$. 
   Then there
   exists an open path $\pi'$ in $G_p$ connecting $x$ to $y$ which lies in
   \[
      \left( \pi \setminus \hat{\rho}^{-1}(\hat{\gamma}) \right)  \cup \bigcup_{\hat{v} \in \hat{\gamma}'} B(v, R).
   \]
\end{prop}
\begin{proof}
   Denote by $\hat{u}$ and $\hat{v}$ the starting and ending points of $\hat{\gamma}$.
   Let $\pi_1$ be the subpath of $\pi$ which starts from $x$ and ends at the first vertex $u'$ of $\pi$
   which lies inside $\hat{u} \subset V(G)$. Let $\pi_2$ be the subpath of $\pi$ which
   starts from the last vertex $v'$ of $\pi$ which lies inside $\hat{v} \subset V(G)$ and ends
   at $y$. By assumption, we have $d_{\hat{G}}(\hat{\rho}(x), \hat{u}), d_{\hat{G}}(\hat{v}, \hat{\rho}(y)) \ge K'$,
   which implies (by our definition of $K'$) that $d_G(x, u), d_G(v,y) \ge R$.
   This means in particular that the connected component of $\pi_1 \cap B(u,R)$ which contains $u' \in \hat{u} \subset B(u,R/5)$
    intersects both $\bnd B(u,R/2)$
   and $B(u,R/5)$; since $\hat{u}$ is open by assumption, this means that $u' \in \kappa(u)$.
   Similarly, we see that $v' \in \kappa(v)$.
   Then, since $\hat{\gamma}'$ is an open path in $\hat{G}$ connecting $\hat{u}$ to $\hat{v}$,
   by Proposition \ref{prop:macrotomicro} there exists an open path $\gamma'$
   in $(\bigcup_{\hat{x} \in \hat{\gamma}'} B(x,R)) \cap G_p$ connecting $u'$ to $v'$.
   Taking $\pi' := \pi_1 * \gamma' * \pi_2$ then gives the desired path (See Figure \ref{fig:macrogluing}). 
      \begin{figure}
   \input{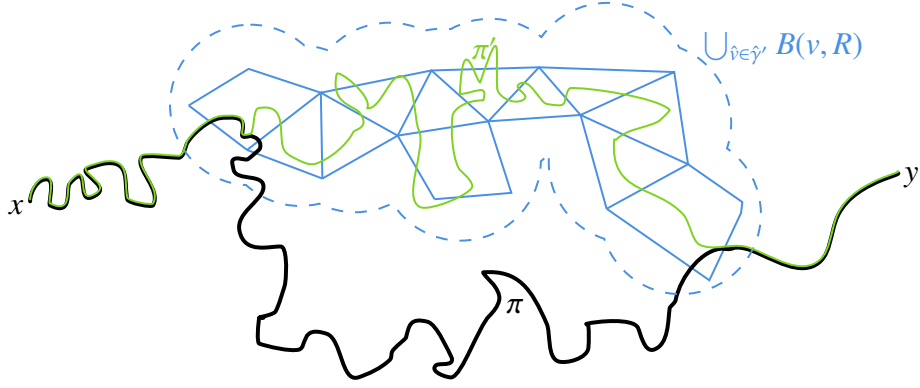}
  \caption{Proposition \ref{prop:surgery}: Gluing a microscopic open path along a macroscopic open path. \label{fig:macrogluing}}
   \end{figure}
\end{proof}

The final ingredient of the proof of Lemma \ref{lem:geolem} is the following geometric property
of coarsely simply connected graphs, which
tells us that we can ``avoid obstacles while staying close to them''.
This lemma is at the heart of our geometric construction, and shows 
how coarse simple-connectedness is key to our argument.
\begin{lemma} \label{lem:obstacles}
   Let $G = (V,E)$ be a $\Delta$-simply connected graph.
   Let $x, y \in V$, let $\beta \subset V$ be a path
   from $x$ to $y$ in $G$, and let $F \subset V$ be a finite
   set of vertices (which we think of as ``forbidden'').
   If there exists some path $\gamma$ from $x$ to $y$ in $G \setminus F$,
   then there exists a path $\gamma'$ from $x$ to $y$ which lies in
   $(N(F, \Delta) \cup \beta) \setminus F$.
   Here $N(F, \Delta) := \bigcup_{f \in F} B(f, \Delta)$
   is the $\Delta$-neighborhood of $F$.
\end{lemma}
\begin{proof}
   Considering $\beta$ and $\gamma$ as edge paths from $x$ to $y$, we have that $\beta \oplus \gamma$ is a cycle.
   Then, since $G$ is $\Delta$-simply connected, there exists a finite family $\mathcal{C}$ of cycles
   such that $\diam c \le \Delta$ for each $c \in \mathcal{C}$ and 
   $\beta \oplus \gamma = \bigoplus_{c \in \mathcal{C}} c$.
   Define $\mathcal{C}(F) := \{ c \in \mathcal{C} : c \cap F \ne \emptyset \}$, the subcollection consisting
   of those small cycles which intersect\footnote{Here, we say that the \emph{edge} set $c$ ``intersects'' the vertex set $F$
   if some edge in $c$ has some endpoint lying in $F$.} 
   the forbidden region. Now define
   \[
      \gamma'' := \beta \oplus \bigoplus_{c \in \mathcal{C}(F)} c.
   \]
   Note that, by definition of $\mathcal{C}$, we then also have that
   \[
      \gamma'' = \gamma \oplus \bigoplus_{c \in \mathcal{C} \setminus \mathcal{C}(F)} c.
   \]
   Our definition of $\gamma''$ tells us that
   \[
      \gamma'' \subset \beta \cup \bigcup_{c \in \mathcal{C}(F)} c \subset \beta \cup N(F, \Delta),
   \]
   where the second inclusion comes from the fact that every $c \in \mathcal{C}(F)$ has diameter at most $\Delta$
   and intersects $F$.
   
   On the other hand, our second formula for $\gamma''$ tells us that
   \[
      \gamma'' \subset \gamma \cup \bigcup_{c \in \mathcal{C} \setminus \mathcal{C}(F) } c \subset V \setminus F,
   \]
   since $\gamma$ does not intersect $F$ by assumption and all $c \in \mathcal{C} \setminus \mathcal{C}(F)$
   do not intersect $F$ by definition.
   
   Now, considering $\gamma''$ as an $\F_2$-valued 1-chain and using the fact
   that the boundary map is $\F_2$-linear and every cycle by definition has $0$ boundary, we have that
   \[
      \bnd \gamma'' = \bnd \gamma \oplus \bigoplus_{c \in \mathcal{C}} c = \{x,y\}.
   \]
   This implies that one of the connected components $\gamma'$ of $\gamma''$ is a path from $x$ to $y$
   (while the other connected components, if they exist, are all cycles). The above inclusions we 
   proved about $\gamma''$ then show that $\gamma'$ has the desired property (see Figure \ref{fig:obstacles})
   \[
      \gamma' \subset (\beta \cup N(F, \Delta)) \setminus F.
   \]
   
   \begin{figure} 
   \input{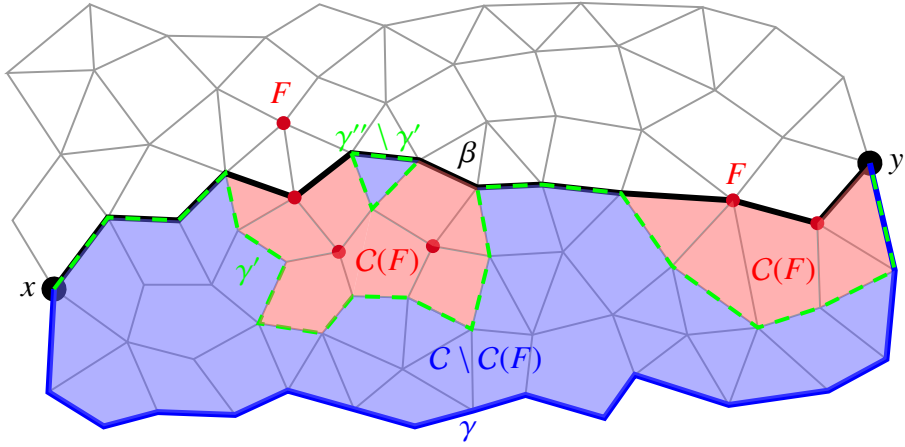}
  \caption{Lemma \ref{lem:obstacles}: Constructing path in $\Delta$-neighborhood of $F$.\label{fig:obstacles}}
   \end{figure}
\end{proof}

Finally we can prove Lemma \ref{lem:geolem}:
\begin{proof}[Proof of Lemma \ref{lem:geolem}]
   Let $\pi$ be an open path in $G_p$ from $x$ to $y$.
   Consider the set
   \[
      X(\pi) := \pi \setminus \left( \bigcup_{ \hat{v} \in I_{x,y} \cup N(F,\Delta) } B(v,R) \right)
   \]
   of sites of $\pi$ which lie outside of the desired region.
   If $X(\pi) = \emptyset$, then we are done.
   Otherwise $X(\pi)$ is a finite union of connected subpaths of $\pi$.
   In this case we will construct another open path $\pi'$ in $G_p$ which connects $x$ to $y$
   such that $X(\pi')$ is \emph{strictly} contained in $X(\pi)$ (and in fact the number of connected components
   will be reduced by at least 1). Then we will be done by induction, since repeating this process
   will reduce $X(\pi')$ to $\emptyset$ in a finite number of steps.
   
   So assume $X(\pi)$ is nonempty, and let $\gamma$ be a connected subpath of $X(\pi)$.
   Consider the macroscopic path $\hat{\pi} := \hat{\rho}(\pi)$.
   We see that $\hat{\rho}(\gamma)$ is a subpath of $\hat{\pi}$ which necessarily lies in 
   $(F \cup B_{\hat{G}}(\hat{\rho}(x),K') \cup B_{\hat{G}}(\hat{\rho}(y),K'))^c$.
   Therefore we can construct a subpath $\hat{\beta}$ of $\hat{\pi}$ containing
   $\hat{\rho}(\gamma)$ as follows. 
   
   Let $\hat{u}'$ be the last site of $\hat{\pi}$ before $\hat{\pi}(\gamma)$ which lies in $F \cup B_{\hat{G}}(\hat{\rho}(x),K' - 1)$,
   and let $\hat{u}$ be the immediately following site in $\hat{\pi}$.
   %If all sites in $\hat{\pi}$ before $\hat{\p}(\gamma)$ lie in $F^c$, or if all such sites which lie in $F$
   %also lie in $ take $\hat{u}$
   %to be the last site of $\hat{pi}$ which lies in $
   Let $\hat{v}'$ be the first site of $\hat{\pi}$ after $\hat{\pi}(\gamma)$ which lies in $F \cup B_{\hat{G}}(\hat{\rho}(y),K' -1)$,
   and let $\hat{v}$ be the immediately preceding site in $\hat{\pi}$.
   Then we define $\hat{\gamma}$ to be the subpath of $\hat{\pi}$ starting at $\hat{u}$ and ending at $\hat{v}$.
   Note that $\hat{\gamma}$ contains $\hat{\rho}(\gamma)$ by construction.
   Note also that $\hat{u}$ is necessarily open in $\hat{G}$. To see this, first note that either $\hat{u}' \in F$
   or $\hat{u} \in B_{\hat{G}}(\hat{\rho}(x), K')$. In the first case, if $\hat{u}$ were closed, then it would be
   a closed site of $\hat{G}$ at distance $1 \le \Delta$ from $F$, and hence it would be contained in $F$, as
   part of a closed $\Delta$-cluster intersecting $I_{x,y}$. In the second case, if $\hat{u}$ were closed it would
   be a closed site in $I_{x,y}$, and hence in $F$. In either case, we would have that $\hat{u} \in F$,
   contradicting the definition of $\hat{u}'$. An analogous argument shows that $\hat{v}$ is an open site of
   $\hat{G}$. 
   
   Note now also that by construction $\hat{\gamma}$ is disjoint from $F$ and its endpoints $\hat{u}, \hat{v}$
   satisfy $$d_{\hat{G}}(\hat{\rho}(x), \hat{u}), d_{\hat{G}}(\hat{v},\hat{\rho}(y)) \ge K.$$
   Moreover, we claim that there is a (not necessarily open) path $\hat{\beta}$ in $I_{x,y} \cup N(F,\Delta)$
   which joins $\hat{u}$ to $\hat{v}$. To see this, note that either $\hat{u}' \in F$ or $\hat{u} \in B_{\hat{G}}(\hat{\rho}(x), K')$.
   In the first case, $\hat{u}'$ is in the closed $\Delta$-component of some site $\hat{z} \in I_{x,y}$, and we claim that
   therefore there exists a path in $\hat{G}$ from $\hat{u}'$ to $\hat{z}$ which lies in $N(F,\Delta)$.
   By definition there is a $\Delta$-path (that is, a sequence of vertices with each consecutive pair having distance
   at most $\Delta$) in $F$ starting at $\hat{u}'$ and ending at $\hat{z}$; connecting the consecutive pairs
   in the $\Delta$-path by geodesics in $\hat{G}$ gives the desired path from $\hat{u}'$ to $\hat{z}$.
   Since $\hat{u}$ is adjacent to $\hat{u'}$ we then immediately obtain a path from $\hat{u}$ to $\hat{z} \in I_{x,y}$
   which lies in $N(F, \Delta)$. 
   On the other hand, if $\hat{u}' \notin F$, we have that $\hat{u} \in I_{x,y}$, so in either case we have that
   $\hat{u}$ is connected to $I_{x,y}$ by a (not necessarily open) path in $N(F,\Delta) \cup I_{x,y}$.
   Similar considerations show the same for $\hat{v}$. Since $I_{x,y}$ is itself connected in $\hat{G}$,
   we then obtain a (not necessarily open) path $\hat{\beta}$ from $\hat{u}$ to $\hat{v}$
   which lies inside $N(F,\Delta) \cup I_{x,y}$.
   
   Therefore, the hypotheses of Lemma \ref{lem:obstacles} hold: $\hat{\beta}$
   is a path from $\hat{u}$ to $\hat{v}$, and $\hat{\gamma}$ is a path from $\hat{u}$ to $\hat{v}$ which lies in $F^c$.
   Therefore by Lemma \ref{lem:obstacles} there exists a path $\hat{\gamma}'$ from $\hat{u}$ to $\hat{v}$ which lies in 
   $(N(F,\Delta) \cup \beta) \setminus F \subset (N(F, \Delta) \cup I_{x,y}) \setminus F$.
   We now claim that $\hat{\gamma}'$ is necessarily open. This is again because if $\hat{z} \in I_{x,y} \cup N(F, \Delta)$
   is a closed site, it necessarily belongs to $F$ by definition.
   
   But now the hypotheses of Proposition \ref{prop:surgery} hold: $\hat{u}$ and $\hat{v}$ are open sites in $\hat{\pi}$
   with $d_{\hat{G}}$ distance at least $K'$, and $\hat{\gamma}'$ is an open path of sites in $\hat{G}$ joining them.
   Therefore we can replace the microscopic open path $\pi$ with another open path $\pi'$ from $x$ to $y$
   which is contained in $(\pi \setminus \hat{\rho}^{-1}(\hat{\gamma})) \cup \bigcup_{\hat{v} \in \hat{\gamma}'} B_G(v,R)$.
   Since we ensured $\hat{\rho}(\gamma) \subset \hat{\gamma}$ and $\hat{\gamma}' \subset N(F, \Delta) \cup I_{x,y}$,
   we have that
   \[
      X(\pi') \subset X(\pi) \setminus \gamma,
   \]
   that is, we have strictly reduced the amount of the path lying outside the desired region, as desired.
\end{proof}

\begin{rmk}
   Note that our approach here differs from that of \cite{AP} in that they seek to find a single open macroscopic path
   from a site near $x$ to a site near $y$ (which can then be microscopically glued to $x$ and $y$).
   This approach necessitates some finnicky case-checking (see e.g. Lemma 3.2 of \cite{AP})
   to deal with the possibility that $x$ and $y$ are ``cut off'' from one another by macroscopic closed components.
   Worse still, starting from both ends, one may a priori have to travel through arbitrarily many
   nested macroscopic closed components before one reaches a common open macroscopic component.
   
   On the other hand, our method capitalizes on the fact that our goal is not actually to avoid macroscopically closed sites,
   but just to find a microscopic open path which stays close to $F \cup I_{x,y}$.
   Therefore there is no reason to modify parts of our original path $\pi$ which pass through $F$; indeed, these
   are microscopically open and close to $F$. Instead we just modify the parts of the path between exits from
   and entries to $F$ so that they lie close to $F \cup I_{x,y}$,
   and so we do not have to treat the cases that $x$ or $y$ are nested within macroscopically closed clusters separately.
   %. This surgery does of course make use of
   %macroscopic open paths, but since we do not need to use a single macroscopic open path to make all modifications to $\pi$ at once,
   %we do not have to consider separately the cases that $x$ or $y$ are nested within 
   %macroscopically closed clusters.
\end{rmk}
%      \begin{figure}
%   \input{Lemma1}
%  \caption{Constructing bypath \label{fig:lem1}}
%   \end{figure}
%    %%%%%%%%%%%%%%%%%
%    \begin{figure}
%   \input{nested}
%  \caption{Nested case \label{fig:lem1nested}}
%   \end{figure}
%%%%%%%%%%%%%%%%%%%%%%%%%%
%%%%%%%%%%%%%%%%%%%%%%%%%%
      \begin{figure}
   % !TEX root =chemical distance polynomial growth.tex
\begin{tikzpicture}[scale=0.5]

% Hexagon parameters
\def\l{1} % side length of hexagon
\def\dx{1.5} % x spacing between centers
\def\dy{0.8660254} % y spacing (sin(60°) * 1.5)

% One hexagon with 6 individual line segments
% Centered at (#1,#2)
\newcommand{\drawhex}[2]{
  \pgfmathsetmacro{\x}{#1}
  \pgfmathsetmacro{\y}{#2}
  \foreach \a in {0,60,...,300} {
    \pgfmathsetmacro{\xA}{\x + \l*cos(\a)}
    \pgfmathsetmacro{\yA}{\y + \l*sin(\a)}
    \pgfmathsetmacro{\xB}{\x + \l*cos(\a+60)}
    \pgfmathsetmacro{\yB}{\y + \l*sin(\a+60)}
    \draw[black] (\xA,\yA) -- (\xB,\yB);
  }
}

\newcommand{\fillhex}[3]{
  \pgfmathsetmacro{\x}{#2}
  \pgfmathsetmacro{\y}{#3}
  \pgfmathsetmacro{\xA}{\x + \l*cos(0)}
  \pgfmathsetmacro{\yA}{\y + \l*sin(0)}
  \pgfmathsetmacro{\xB}{\x + \l*cos(60)}
  \pgfmathsetmacro{\yB}{\y + \l*sin(60)}
  \pgfmathsetmacro{\xC}{\x + \l*cos(120)}
  \pgfmathsetmacro{\yC}{\y + \l*sin(120)}
  \pgfmathsetmacro{\xD}{\x + \l*cos(180)}
  \pgfmathsetmacro{\yD}{\y + \l*sin(180)}
  \pgfmathsetmacro{\xE}{\x + \l*cos(240)}
  \pgfmathsetmacro{\yE}{\y + \l*sin(240)}
  \pgfmathsetmacro{\xF}{\x + \l*cos(300)}
  \pgfmathsetmacro{\yF}{\y + \l*sin(300)}
  \path[fill=#1!30] 
    (\xA,\yA) -- (\xB,\yB) -- (\xC,\yC)
    -- (\xD,\yD) -- (\xE,\yE) -- (\xF,\yF)
    -- cycle;
    \drawhex{\x}{\y}
}
%%%%%%%%%%%%%%%%%%%%%%%%%%%%%%%%%%%%%%
%%%%%%%%%%%%%%%%%%%%%
%%%%%\begin{scope}
%%%%%  % Define centers of filled hexagons
%%%%%  \coordinate (A) at (\dx, \dy);
%%%%%  \coordinate (B) at (8*\dx, 2*\dy);
%%%%%  \coordinate (C) at (8*\dx, 4*\dy);
%%%%%  \coordinate (D) at (9*\dx, 5*\dy);
%%%%%  \coordinate (E) at (9*\dx, 7*\dy);
%%%%%  \coordinate (F) at (10*\dx, 6*\dy);
%%%%%  \coordinate (G) at (8*\dx, 6*\dy);
%%%%%  \coordinate (H) at (7*\dx, 5*\dy);
%%%%%  \coordinate (I) at (7*\dx, 3*\dy);
%%%%%  \coordinate (J) at (6*\dx, 2*\dy);
%%%%%  \coordinate (K) at (6*\dx, 0);
%%%%%  \coordinate (L) at (10*\dx, 4*\dy);
%%%%%  \coordinate (M) at (9*\dx, 3*\dy);
%%%%%  \coordinate (N) at (9*\dx, \dy);
%%%%%  \coordinate (O) at (8*\dx, 0);
%%%%%  \coordinate (P) at (7*\dx, -\dy);
%%%%%  \coordinate (Q) at (2*\dx, 0);
%%%%%  \coordinate (R) at (3*\dx, \dy);
%%%%%  \coordinate (S) at (2*\dx, 4*\dy);
%%%%%  \coordinate (T) at (\dx, 3*\dy);
%%%%%  \coordinate (U) at (3*\dx, 2*\dy);
%%%%%  \coordinate (V) at (3*\dx, 3*\dy);
\coordinate (H01) at (0, 0);
\coordinate (H02) at (\dx, \dy);
\coordinate (H03) at (\dx, -\dy);
%\coordinate (H04) at (3*\dx, -\dy);
\coordinate (H05) at (7*\dx, -\dy);
\coordinate (H06) at (6*\dx, 0);
\coordinate (H07) at (6*\dx, 2*\dy);

\coordinate (H08) at (7*\dx, \dy);
\coordinate (H09) at (8*\dx, 2*\dy);
\coordinate (H10) at (8*\dx, 4*\dy);
\coordinate (H11) at (9*\dx, 5*\dy);
\coordinate (H12) at (9*\dx, 7*\dy);

\coordinate (H13) at (10*\dx, 6*\dy);
\coordinate (H14) at (8*\dx, 6*\dy);
\coordinate (H15) at (7*\dx, 5*\dy);
\coordinate (H16) at (7*\dx, 3*\dy);
\coordinate (H17) at (6*\dx, 2*\dy); % Already defined above
\coordinate (H18) at (6*\dx, 0);     % Already defined above
\coordinate (H19) at (10*\dx, 4*\dy);
\coordinate (H20) at (9*\dx, 3*\dy);
\coordinate (H21) at (9*\dx, \dy);
\coordinate (H22) at (8*\dx, 0);
\coordinate (H23) at (7*\dx, -1*\dy);

\coordinate (H24) at (8*\dx, 0);      % duplicate
\coordinate (H25) at (9*\dx, \dy);    % duplicate
%\coordinate (H26) at (10*\dx, 0);    % commented
\coordinate (H27) at (10*\dx, 2*\dy);
\coordinate (H28) at (11*\dx, \dy);
%\coordinate (H29) at (12*\dx, 0);    % commented
\coordinate (H30) at (12*\dx, 2*\dy);
\coordinate (H31) at (13*\dx, \dy);
\coordinate (H32) at (14*\dx, 0);
\coordinate (H33) at (14*\dx, 2*\dy);
\coordinate (H34) at (15*\dx, \dy);
\coordinate (H35) at (15*\dx, 3*\dy);
\coordinate (H36) at (15*\dx, -1*\dy);
\coordinate (H37) at (16*\dx, 0);
\coordinate (H38) at (16*\dx, 2*\dy);

\coordinate (H39) at (2*\dx, 0);
\coordinate (H40) at (3*\dx, \dy);
\coordinate (H41) at (2*\dx, 4*\dy);
%\coordinate (H42) at (4*\dx, 0);
\coordinate (H43) at (5*\dx, \dy);

\coordinate (H44) at (0, 2*\dy);
\coordinate (H45) at (\dx, 3*\dy);
\coordinate (H46) at (3, 2*\dy);
\coordinate (H47) at (3*\dx, 3*\dy);
\coordinate (H48) at (4*\dx, 2*\dy);
%\coordinate (H49) at (5*\dx, 3*\dy);
%\coordinate (H50) at (0, 4*\dy);

%  % Create union of balls of radius 1.5*\dx
%  \begin{scope}
%    \clip
%      (A) circle (1.3*\dx)
%      (B) circle (1.3*\dx)
%      (C) circle (1.3*\dx)
%      (D) circle (1.3*\dx)
%      (E) circle (1.3*\dx)
%      (F) circle (1.3*\dx)
%      (G) circle (1.3*\dx)
%      (H) circle (1.3*\dx)
%      (I) circle (1.3*\dx)
%      (J) circle (1.3*\dx)
%      (K) circle (1.3*\dx)
%      (L) circle (1.3*\dx)
%      (M) circle (1.3*\dx)
%      (N) circle (1.3*\dx)
%      (O) circle (1.3*\dx)
%      (P) circle (1.3*\dx)
%      (Q) circle (1.3*\dx)
%      (R) circle (1.3*\dx)
%      (S) circle (1.3*\dx)
%      (T) circle (1.3*\dx)
%      (U) circle (1.3*\dx)
%      (V) circle (1.3*\dx);
%    \fill[white] (0,0) rectangle (20*\dx, 10*\dy); % white fill to intersect clip paths
%  \end{scope}
%
  % Draw the boundary
 
\foreach \pt in {H01,H02,H03,H05,H06,H07,H08,H09,H10,H11,H12,H13,H14,H15,H16,H17,H18,H19,H20,H21,H22,H23,H27,H28,H30,H31,H32,H33,H34,H35,H36,H37,H38,H39,H40,H41,H43,H44,H45,H46,H47,H48} {
    \draw[blue, dashed] (\pt) circle (1.3*\dx);
}

\foreach \pt in {H01,H02,H03,H05,H06,H07,H08,H09,H10,H11,H12,H13,H14,H15,H16,H17,H18,H19,H20,H21,H22,H23,H27,H28,H30,H31,H32,H33,H34,H35,H36,H37,H38,H39,H40,H41,H43,H44,H45,H46,H47,H48} {
   \fill[white] (\pt) circle (1.3*\dx);
}

% \foreach \pt in {H01,H02,...,H50} {
%  \draw[dashed, thick, purple] (\pt) circle (1.3*\dx);
%}
  
%  \foreach \pt in {H01,H02,...,H50} {
%  \fill[white] (\pt) circle (1.3*\dx);
%}
% 
% \foreach \point in {A,B,C,D,E,F,G,H,I,J,K,L,M,N,O,P,Q,R,S,T,U,V}{
%    \draw[dashed, thick, purple] (\point) circle (1.3*\dx);
%  }
%  \foreach \point in {A,B,C,D,E,F,G,H,I,J,K,L,M,N,O,P,Q,R,S,T,U,V}{
%  \fill[white] (\point) circle (1.3*\dx);
%}

%\end{scope}

%%%%%%%%%%%%%%%%%%%%%%%%%%%%%%%%%%%%%%%%%%%%%%
%%%%%%%%%%%%%%%%%%%%%%%%%%%%%%%%%%%%%%%%%%%%%
% First few rows of the lattice (manually expanded)
\drawhex{0}{0}
\fillhex{blue}{\dx}{\dy}
\drawhex{\dx}{-\dy}
%\drawhex{3*\dx}{-\dy}
\drawhex{7*\dx}{-\dy}
\drawhex{6*\dx}{0}
\drawhex{6*\dx}{2*\dy}
\fillhex{red}{7*\dx}{\dy}
\fillhex{red}{8*\dx}{2*\dy}
\fillhex{red}{8*\dx}{4*\dy}
\fillhex{red}{9*\dx}{5*\dy}
\fillhex{blue}{9*\dx}{7*\dy}

\fillhex{blue}{10*\dx}{6*\dy}
\fillhex{blue}{8*\dx}{6*\dy}
\fillhex{blue}{7*\dx}{5*\dy}
\fillhex{blue}{7*\dx}{3*\dy}
\fillhex{blue}{6*\dx}{2*\dy}
\fillhex{blue}{6*\dx}{0*\dy}
\fillhex{blue}{10*\dx}{4*\dy}
\fillhex{blue}{9*\dx}{3*\dy}
\fillhex{blue}{9*\dx}{1*\dy}
\fillhex{blue}{8*\dx}{0*\dy}
\fillhex{blue}{7*\dx}{-1*\dy}

\drawhex{8*\dx}{0}
\drawhex{9*\dx}{\dy}
%\drawhex{10*\dx}{0}
\drawhex{10*\dx}{2*\dy}
\drawhex{11*\dx}{\dy}
%\drawhex{12*\dx}{0}
\drawhex{12*\dx}{2*\dy}
\drawhex{13*\dx}{\dy}
\drawhex{14*\dx}{0}
\drawhex{14*\dx}{2*\dy}
\drawhex{15*\dx}{\dy}
\drawhex{15*\dx}{3*\dy}
\drawhex{15*\dx}{-1*\dy}
\drawhex{16*\dx}{0}
\drawhex{16*\dx}{2*\dy}

\fillhex{blue}{2*\dx}{0}
\fillhex{blue}{3*\dx}{\dy}
\drawhex{2*\dx}{4*\dy}
\fillhex{blue}{2*\dx}{4*\dy}
%\drawhex{4*\dx}{0}
\drawhex{5*\dx}{\dy}

\drawhex{0}{2*\dy}
\fillhex{blue}{\dx}{3*\dy}
\fillhex{red}{3}{2*\dy}
\fillhex{blue}{3*\dx}{3*\dy}
\drawhex{4*\dx}{2*\dy}
%\drawhex{5*\dx}{3*\dy}

%\drawhex{0}{4*\dy}

% Black Curve connecting x to y
 \draw[line width=1pt, black, smooth]
  plot coordinates {
    (1.3*\dx, 1.5*\dy)
    (1.7*\dx, 1.5*\dy)
    (2*\dx, 4.5*\dy)
    (3*\dx, 5.5*\dy)
    (4*\dx, 9*\dy)
    (5*\dx, 8*\dy)
    (5*\dx, 6.5*\dy)
    (6*\dx, 5.2*\dy)
    (7*\dx, 4.7*\dy)
    (8*\dx, 4.2*\dy)
    (9*\dx, 3.6*\dy)
    (10*\dx, 2.5*\dy)
    (11*\dx, -1*\dy)
    (11.5*\dx, 0*\dy)
    (12*\dx, -0.5*\dy)
    (12.5*\dx, -2*\dy)
    (13*\dx, -3*\dy)
    (13.5*\dx, -1.2*\dy)
    (14*\dx, -0.8*\dy)
    (14.3*\dx, -2*\dy)
    (14.6*\dx, -0.7*\dy)
    (14.9*\dx, 1.2*\dy)
  };
  
   \draw[line width=1pt, green, smooth]
  plot coordinates {
    (2*\dx, 4.5*\dy)
   (2.5*\dx, 3.5*\dy)
    (4*\dx, 3.5*\dy)
    (5*\dx, 2*\dy)
    (5.5*\dx, 3.3*\dy)
    (5.8*\dx, 3.8*\dy)
    (6*\dx, 5.2*\dy)
  };

\draw (1.1*\dx,1.5*\dy) node {$x$};
\draw (14.9*\dx,1.5*\dy) node {$y$};

\draw (5.1*\dx,5.1*\dy) node {$\pi$};
\draw (2.1*\dx,4*\dy) node {$\hat{u}$};
\draw (2.1*\dx,2.1*\dy) node {$\hat{u}'$};

\draw (8*\dx,3.6*\dy) node {$\hat{v}'$};
\draw (7*\dx,5.2*\dy) node {$\hat{v}$};

\end{tikzpicture}
\caption{Lemma \ref{lem:geolem}: modifying a part of $\pi$ to stay near $I_{x,y} \cup N(F, \Delta)$. \label{fig:lem1}}
   \end{figure}
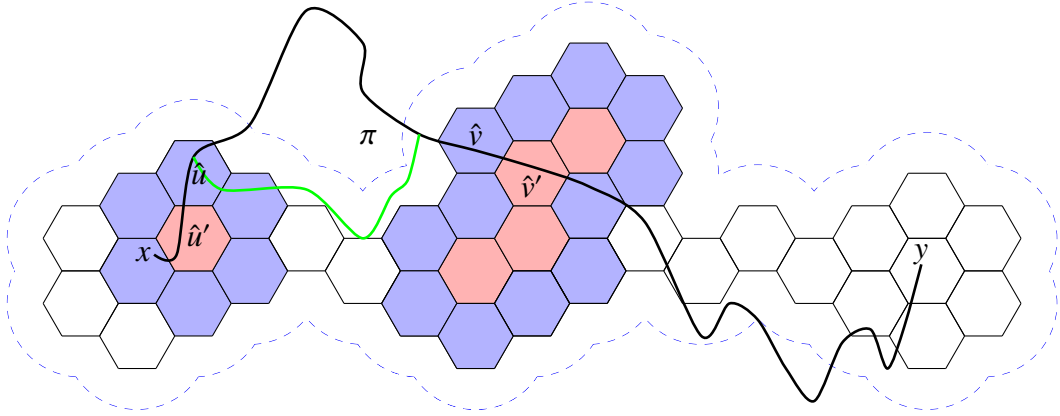

\subsection{Proof of Lemma \ref{lem:problem}} \label{sec:proof_of_lem_problem}

The purpose of this section is to prove Lemma \ref{lem:problem}, i.e. that any fixed sufficiently large scale,
$|F|$ is exponentially likely to be at most $K d_G(x,y)$.
We do this by showing that $|F|$ is stochastically dominated by a sum of at most $d_G(x,y) + O(1)$ independent random
variables with uniform exponential bounds on their tails.

First, to relate our dependent site percolation on $\hat{G}$ to an independent percolation,
we use the famous ``domination by product measures'' result of Liggett, Schonmann, and Stacey.
\begin{thm}[\cite{LSS} Corollary 1.4] \label{thm:LSS}
   Fix $\hat{D}, K' < \infty$. Then there exists a function $\sigma: [0,1] \to [0,1]$ (depending on $\hat{D}$ and $K'$)
   such that $\lim_{q \to 1} \sigma(q) = 1$
   with the following property. Let $\hat{G}$ be any graph with degree at most $\hat{D}$ and let
   $X$ be a random subset of sites of $\hat{G}$ which has ``dependence range $K'$'' in the sense that
   for any $K'$-separated subset $S \subset \hat{G}$, the collection of events
   \[
      \left \{ \{ s \in X \} : s \in S \right\}
   \]
   is mutually independent.
   Suppose also that for some $q \in [0,1]$, $\Prob( v \in X ) \ge q$ for all $v \in \hat{V}$.
   Then there exists a coupling of $X$ and $\hat{G}_{\sigma(q)}$ such that
   \[
      \hat{G}_{\sigma(q)} \subset X
   \]
   almost surely,
   where $\hat{G}_{\sigma(q)}$ is independent Bernoulli site percolation on $\hat{G}$ with parameter 
   $\sigma(q)$.
\end{thm}

Now note that $|F|$ is a decreasing random variable, in the sense that if we consider two 
configurations of open sites $X, Y \subset \hat{V}$ such that $X \subset Y$, then the value of $|F|$
given by the configuration $Y$ is at most as large as that given by the configuration $X$. 
Therefore, for a fixed scale $R$, setting $q = \Prob_p(A_v) = \Prob_p( \hat{v} \mbox{ is open})$,
we have that
\[
   \Prob_p( |F| \ge K d_G(x,y) ) \le \Prob( |F_{\sigma(q)}| \ge K d_G(x,y) )
\]
where we define $F_{\sigma(q)}$ to be the union of all the closed $\Delta$-clusters intersection $I_{x,y}$
\emph{with respect to} $\hat{G}_{\sigma(q)}$, the Bernoulli site percolation on $\hat{G}$ with
parameter $\sigma(q)$. Thus, it suffices to show that the latter quantity decays exponentially in $d_G(x,y)$.

\begin{prop} \label{prop:indepbound}
   There exists $0 < \rho_0 < 1$ (depending only on $G$, not on $R$ or $\hat{G}$)
   and $K, C < \infty, c > 0$ 
   such that for any $\rho_0 < \rho < 1$, for all $t \ge d_G(x,y)$,
   \[
      \Prob( |F_{\rho}| \ge K t ) \le C \exp(-c t)
   \]
   for all $x,y \in V$.
\end{prop}
\begin{proof}
   For each site $\hat{v} \in \hat{G}$, let $\tilde{C}(\hat{v})$ be a random subset of $\hat{V}$
   which has the same law as the closed $\Delta$-cluster of $\hat{v}$  in $\hat{G}_{\rho}$ but is
   independent from the process $\hat{G}_{\rho}$ and all the other $\tilde{C}(\hat{w})$.
   These are sometimes called \emph{preclusters} (compare this
   to \cite{AP}, Section 4, page 1047, in the midst of the proof of Theorem 1.1).
   We then have that $|F_{\rho}|$ is stochastically dominated by
   the sum of independent random variables $\sum_{\hat{v} \in I_{x,y}} |\tilde{C}(\hat{v})|$.
   %\textcolor{red}{Probably better to just cite, but: To see this,
   %note that one can sample a set with the same distribution as $F_{\rho}$ as follows.
   %Fix an arbitrary enumeration $\{\hat{v}_0,...,\hat{v}_N\}$ of $I_{x,y}$.
   %First define $F_0 = \tilde{C}(\hat{v}_0)$.
   %Then, given $F_0,...,F_i$, define $F_{i+1} = \emptyset$ if $\hat{v}_i \in \bigcup_{j=0}^i N(F_j, \Delta)$
   %and otherwise define $F_{i+1}$ to be the connected component of
   %$\tilde{C}(\hat{v}_{i+1}) \setminus \bigcup_{j=0}^i N(F_j, \Delta)$ which contains $\hat{v}_{i+1}$.
   %One can check that then $F_N$ has the same law as $F_{\rho}$, but we
   %have that $|F_N| \le \sum_{\hat{v} \in I_{x,y}} |\tilde{C}(\hat{v})|$.}
   
   Moreover, we have a uniform exponential tail bound on the collection $|\tilde{C}(\hat{v})|$
   if we take $\rho > \rho_0 := 1 - (2 D^{\Delta})^{-1}$; we see that 
   \begin{align*}
      \Prob( \tilde{C}(\hat{v}) = k ) &\le 
      \sum_{\substack{\hat{v} \in S \subset \hat{V}, |S| = k, \\ 
      S \mbox{ } \Delta-\mbox{connected}}} \Prob_{\rho}( S \mbox{ closed}) \\
      &\le (2 D^{\Delta})^k (1 - \rho)^k
   \end{align*}
   decays exponentially in $k$. Here in the second inequality we have used the fact that the number of $\Delta$-connected
   subsets of size $k$ containing $\hat{v}$ is at most $(2 D^{\Delta})^k$; since $\Delta$-connected subsets are
   precisely subsets which are connected in the graph which connects any two vertices of $\hat{G}$ with
   distance at most $\Delta$ by an edge, and this graph has degree at most $D^{\Delta}$, this observation
   follows from the observation that for any connected graph, there exists a path starting from any vertex
   which covers a spanning tree but traverses each edge at most twice.
   
   The exponential tail bound then gives us an exponential moment, that is,
   fixing $0 < s < - \log [ 2 D^{\Delta} (1 - \rho) ]$, we have that
   $\E[ \exp( s |\tilde{C}(v)| ) ] \le \sum_{k=0}^{\infty} (e^s 2 D^{\Delta} (1 - \rho) )^k =: Q < \infty$.
   
   Now choose $K$ sufficiently large that $\log Q - sK < 0$.
   Then for any $t \ge d_G(x,y)$ we obtain the large deviation bound
   \begin{align*}
      \Prob( |F_\rho| \ge K t ) &\le
      \Prob\left( \sum_{\hat{v} \in I_{x,y}} |\tilde{C}(\hat{v})| \ge K t \right) \\
      &= \Prob\left( \prod_{\hat{v} \in I_{x,y}} \exp( s |\tilde{C}(\hat{v})| ) \ge \exp( s K t ) \right) \\
      &\le \exp(-sK t) \prod_{\hat{v} \in I_{x,y}} \E\left[ \exp(t |\tilde{C}(\hat{v}) | )\right] \\
      &\le \exp(-sK t) Q^{|I_{x,y}|} \\
      &\le Q^{2 \hat{D}^{K'}} \exp( [(\log Q) - sK ] t ),
   \end{align*}
   where we have used independence and Markov's inequality in the third line.
   By our choice of $K$, this quantity is decreasing exponentially in $t$,
   and so we are done. %\textcolor{red}{Maybe just replace details with ``use a standard Chernoff bound''.}
\end{proof}

\begin{proof}[Proof of Lemma \ref{lem:problem}]
   Given $G$, take $\rho_0 < \rho < 1$ where $\rho_0$ is as in Proposition \ref{prop:indepbound}.
   Then take $0 < q < 1$ large enough that $\sigma(q) \ge \rho$, where $\sigma$ is given
   by Theorem \ref{thm:LSS}.
   Finally, since $\inf_{\hat{v} \in \hat{G}} \Prob_p(A_v)$ tends to 1 as $R$ tends to infinity
   by Proposition \ref{prop:quantunique} \cite{CMT},
   fix $R$ sufficiently large that $\Prob_p(A_v) \ge q$. Since by Theorem \ref{thm:LSS} the dependent site percolation on $\hat{G}$
   induced by $G_p$ then stochastically dominates the Bernoulli site percolation $\hat{G}_{\rho}$, we have that
   \[
      \Prob( |F| > K d_G(x,y) ) \le \Prob( |F_{\rho}| > d_G(x,y) )
   \]
   decays exponentially in $d_G(x,y)$, as desired.
\end{proof}

Finally, let us quickly sketch the proof of Theorem \ref{thm:finitelypresented}.
The proof is essentially the same as the proof of Theorem \ref{thm:main}, except that we do not
perform any coarse graining; for this reason we have to assume that $p$ is close to 1,
instead of just supercritical. We give the proof for edge percolation;
the proof for site percolation is even more similar to the proof of Theorem \ref{thm:main}.

\begin{proof}[Proof of Theorem \ref{thm:finitelypresented}]
   Let $\Gamma = \langle S | R \rangle$ be a finitely presented group, let $G$ be its Cayley graph associated
   to the generating set $S$, and suppose that $\Delta$ be an upper bound for the word length (with respect to $S$)
   of every relator in $R$ (which we assume to be a finite set).
   Then $G$ has degree at most $2|S|$ and it is $\Delta$-simply connected.
   
   We fix $p \in [0,1]$ and perform independent edge percolation on $G$.
   Say that an edge $e$ in $E(G)$ is $\Delta$-adjacent to the edge $e'$ if there is an edge path in $G$ of diameter less than $\Delta$
   which contains both $e$ and $e'$. 
   We define the $\Delta$-\emph{closed cluster of} $e$ to be empty if $e$ is open
   and otherwise equal to the connected component of $e$ in the graph whose vertex set is the set of closed edges of $G$
   and whose edges connect $\Delta$-adjacent closed edges of $G$.
   
   Now, fix $x,y \in V(G)$, and fix an edge-geodesic $[x,y]$ from $x$ to $y$.
   Define $F$ to be the union of the $\Delta$-closed clusters of all the edges in $[x,y]$.
   Suppose that $x$ is connected to $y$ in $G_p$.
   A version of Lemma \ref{lem:obstacles} then tells us there is an open path in $N(F, \Delta) \cup [x,y]) \setminus F$
   from $x$ to $y$, where here $N(F, \Delta)$ consists of all edges of $G$ which are
   $\Delta$-adjacent to some edge in $F$.
   (Once we interpret $N(F, \Delta)$ this way, the only modification we need to make to Lemma \ref{lem:obstacles}
   is that $F$ is a set of forbidden edges, rather than forbidden vertices, and that the paths in question are
   edge paths rather than vertex paths. The proof of this modified lemma is identical).
   
   Thus, we see that if $x \connectsthru{G_p} y$, then
   \[ d_{G_p}(x,y) \le d_G(x,y) + |N(F, \Delta)| \le d_G(x,y) + 2(2|S|)^{\Delta} |F|, \]
   so again we will be done once we prove that, for some $p_0 < 1$, 
   for every $p \in (p_0, 1]$, $\Prob_p( |F| \ge t )$ decays exponentially in $t$ for $t \ge d_G(x,y)$.
   But the proof of this statement is almost exactly the same as the proof
   of Proposition \ref{prop:indepbound}, the only difference being a slightly different bound
   on the number of $\Delta$-connected sets of size $k$ (again just because we are dealing with edges instead of vertices).
   One can check that the $p_0$ we get only depends on $|S|$ and $\Delta$.
   
   The above gives the first statement of the theorem. The second statement follows from the fact that when 
   a unique infinite cluster exists, we have $\inf_{x,y} \Prob(x \connectsthru{G_p} y) > 0$, just
   as discussed in the beginning of the proof of Theorem \ref{thm:main}.
\end{proof}

\section{Continuity of the time constants} \label{sec:continuity}
Theorem \ref{thm:main} can be interpreted as saying that the intrinsic (``chemical'') distance on an infinite supercritical
cluster is very likely at most a constant times the ambient graph distance.
The Kingman subadditive ergodic theorem \cite{Kingman} implies that along a fixed ``direction'', at large scales
the chemical distance is very likely \emph{very close} to a constant times the ambient graph distance,
where the constant depends on the direction. More precisely, for a point $u \in V$, denote by $\mathring{u}$ the (random) site
of the infinite cluster %$C_p(\infty)$ of 
$G_p$ which is closest in graph distance to $u$ (with ties broken by choosing uniformly 
at random from the finitely many candidates); then define a pseudometric on $V$ by taking
$D_p(x,y) := d_{G_p}(\mathring{x}, \mathring{y})$. Then, for any $g \in \Gamma := \Aut(G)$ there exists a
``time constant'' $\mu_p(g)$ such that
\[
   \lim_{n \to \infty} \frac{D_p(o, g^n o)}{n} = \mu_p(g)
\]
almost surely, and in mean. (Here $o$ is an arbitrary basepoint for $G$).
In fact, the collection of time constants $\mu(g)$ has much more structure;
they give a norm on $\Gamma/[\Gamma,\Gamma] \otimes \R \cong \R^{d'}$,
and associated to this norm is a particular metric space
(a certain nilpotent Lie group with a Carnot-Carath\'{e}odory metric)
which one should expect to be the almost-sure Gromov-Hausdorff scaling limit
of the random pseudometric $D$ (see \cite{CantrellFurman, BenjaminiTessera}).
Since showing this latter result would require 
weakening the assumptions in \cite{CantrellFurman}, we do
not pursue this here.
The thesis \cite{de2024asymptotic}
does attempt such a weakening; once this result is published, it should be possible to combine it with 
the results here in order to deduce a shape theorem in our setting.

In any case, if such a shape theorem holds, the limit geometry
is determined by limits of the form
\[
   \lim_{n \to \infty} \frac{ \E D_p(o, x_n) }{ d_G(o,x_n) }.
\]
A natural question, first addressed when $G$ is $\mathbb{Z}^d$ in \cite{garet2017continuity}, is whether these geometric quantities are continuous in $p$;
in fact we will be able to show this without addressing the question of
whether the limit in question exists---our methods will show Lipschitz continuity
of any $\limsup$ or $\liminf$ uniformly. 
The goal of the rest of this paper is to prove Theorem \ref{thm:cty} as stated in Section \ref{sec:intro}.
%\begin{thm} \label{thm:cty}
%   Above criticality, the time constants are Lipschitz continuous in $p$ in the following sense.
%   For each $p_0 > p_c$, there exists $C(p_0)$ such that
%   \[
%      \limsup_{x \to \infty} \left| \frac{ \E D_p(o, x) }{ d_G(o,x) } - \frac{\E D_q(o,x) }{ d_G(o,x) } \right|
%      \le C(p_0)|p-q|
%   \]
%   for all $p,q \in [p_0, 1]$.
%   In particular, given any sequence $x_n \to \infty$ in $V$, if we define
%   \begin{align*}
%      \overline{d}(p) := \limsup_{n \to \infty} \frac{ \E D_p(o,x_n) }{ d_G(o, x_n) } & &
%      \underline{d}(p) := \liminf_{n \to \infty} \frac{\E D_p(o,x_n) }{ d_G(o,x_n) },
%   \end{align*}
%   then for any $p_0 > p_c$, both $\overline{d}$ and $\underline{d}$ 
%   are Lipschitz functions of $p$ on $[p_0,1]$.
%\end{thm}

The proof of Theorem \ref{thm:cty} is inspired by (and similar to) that of \cite{CNN},
but our setting allows for some simplifications.
First, following an idea from \cite{biskup2015isoperimetry}, we approximate $D_p(o,x)$ by another random variable $\tilde{D}_p(o,x)$ which has the advantages
of being monotone in $p$ and dependent on only finitely many edges.
Then we apply Russo's formula to get an expression for the derivative of $\E \tilde{D}_p(o,x)$; using
geometric conditions which hold with very high probability we will then argue that this is bounded by
\[
   \E \left[\sum_{e = \{u,v\} \in \pi} \ind_{\left\{u \connectsthru{G_p \setminus \{e\} } v \right\}} d_{G_p \setminus \{e\}}(u,v) \right] + \frac{1}{1-p} o(d_G(o,x)),
\]
where $\pi$ is a ``realizing geodesic'' for $\tilde{D}_p(o,x)$, and the summands can
be interpreted as the number of edges needed to ``detour around'' an edge $e$.
%and will see that we can bound its derivative
%by the sum over the edges $e = \{u,v\}$ of part of the ``realizing geodesic'' of the quantity $d_{G_p \setminus \{e\}}(u,v)$
%(which measures the number of edges needed to ``detour'' around $e$), plus some sublinear terms.
Theorem \ref{thm:main} (nearly) tells us that the summands have uniform exponential tails.
Moreover, the summands are only weakly dependent, and length of the realizing geodesic is very likely to be linear;
therefore we can use a greedy lattice animal bound as in \cite{CNN}
to show that the expectation appearing in Russo's formula is at most a constant (independent of $p$ and $x$) times $d_G(o,x)$,
up to lower order terms.
Integrating and then taking the limit in $x$ then gives the desired bound.

The main differences between our approach and the approach of \cite{CNN} are: first, our definition of $\tilde{D}_p$ is
different, only allowing one to use closed edges at the beginning and end of the path, which simplifies some technical arguments;
second, instead of bounding the ``effective radius'' as in \cite{CNN}, we can directly bound the length of
a bypass around an edge using techniques similar to our proof of Theorem \ref{thm:main};
third, the proof that some convenient events are very likely is simplified by using finite-energy type arguments.
The fact that Lemma \ref{lem:linear-bound} holds for general graphs of polynomial growth 
is to our knowledge a new observation; the idea needed to
extend its scope to this general setting is chiefly the coarse-graining procedure itself (which appeared in \cite{Gorski})
and a basic geometric fact about it which is proved in the course of Proposition \ref{prop:separate}
(which did not appear in \cite{Gorski} where this coarse-graining construction was introduced).

\subsection{Approximating $D_p$}
Throughout this section, we assume that $G$ has polynomial growth of degree $d$.
Fix $5 \le C < \infty$.
Fixing a basepoint $o \in V$, for any $x \in V$, let us define the quantity
\[
   \tilde{D}_p(o,x) := \inf_{\substack{o',x' \in V \\ o' \connectsthru{G_p} x' }}  d_{G_p}(o',x')
   +(\log d_G(o,x) )^C (d_G(o,o') + d_G(x',x) ) .
\]
Note that taking $o'=x'=o$, we get the deterministic upper bound
\[
   \tilde{D}_p(o,x) \le  d_G(o,x) (\log d_G(o,x))^C.
\]

Generally, an upper bound for $\tilde{D}_p$ is realized by an open path between a pair
of points which are close to $o$ and $x$ respectively. 
If any point of that open path has distance $d$ from $o$ and $x$, then the associated
upper bound is at least $d$. Therefore, our a priori upper bound also tells us that
the random variable $\tilde{D}_p$ only depends on edges in the finite set of edges
\[
   E(\tilde{D}) := N(\{o,x\},  d_G(o,x) (\log d_G(o,x))^C ) .
\]

We claim $\tilde{D}_p$ is a good approximation to $D_p$:
\begin{prop} \label{prop:goodapprox}
For every $p > p_c$ we have
\[
   \E |D_p(o,x) - \tilde{D}_p(o,x) | = o(d_G(o,x)).
\]
\end{prop}

The rest of the section is dedicated to proving Proposition \ref{prop:goodapprox}.
Most of this will be a consequence of the fact that the random variable
$d_G(o,\mathring{o})$ has rapidly decaying tails, as stated in the following proposition:
\begin{prop} \label{prop:smallholes}
   For any $p > p_c$ we have
   \[
      \Prob_p( d_G(o, \mathring{o}) \ge n) = O( \sqrt{n} \exp(- \sqrt{10n})).
   \]
\end{prop}
\begin{proof}
   Recall $A_v(R)$ as defined in \eqref{eq:goodevent} (here we make explicit the dependence on $R$).
   Note that by ``gluing,'' for any $n \ge 1$ we have 
   \[
      \bigcap_{R=n}^{\infty} A_o(10R) \subset \{ B(o,n) \connects \infty \} = \{ d_G(o, \mathring{o}) \le n \}.
   \]
   Therefore, taking complements and using Proposition \ref{prop:quantunique} we have
   \[
      \Prob( d_G(o, \mathring{o}) \ge n+1) \le \sum_{R=n}^{\infty} \Prob(A_o(10R)^c) 
      \le \sum_{R=n}^{\infty} e^{-\sqrt{10R}} = O( \sqrt{n} \exp(- \sqrt{10n})
   \]
   for sufficiently large $n$.
\end{proof}

We also need the fact that it is very unlikely to have a large component which is not part of the
infinite component:
\begin{thm}[Theorem 1.2 of \cite{CMT}] \label{thm:nofinitegiant}
   For any $p > p_c$, there exists $c>0$ such that for all $n \ge 1$,
   \[
      \Prob_p(n \le |C(o)| \le \infty ) \le \exp(-c n^{\frac{d-1}{d}}).
   \]
\end{thm}

Next, the following proposition tells us that in order to prove Proposition \ref{prop:goodapprox},
we may assume events which are likely as $d_G(o,x) \to \infty$:
\begin{prop} \label{prop:wlog}
   Let $A$ be some event (depending on $x$) such that $\Prob(A^c) = o((\log d_G(o,x))^{-2C})$ as $d_G(o,x) \to \infty$. Then we have
   \[
      \E|D_p(o,x) - \tilde{D}_p(o,x)| = \E [\ind_A |D_p(o,x) - \tilde{D}_p(o,x)|] + o(d_G(o,x)).
   \]
\end{prop}

\begin{proof}
   We have by Cauchy-Schwarz
   \begin{align*}
      \E[ |D_p(o,x) - \tilde{D}_p(o,x)|] - \E[ \ind_A |D_p(o,x) - \tilde{D}_p(o,x)|] 
      &=\E[ \ind_{A^c} |D_p(o,x) - \tilde{D}_p(o,x)| ]\\
      &\le \E[ \ind_{A^c} D_p(o,x) ] + \E[ \ind_{A^c} \tilde{D}_p(o,x)] \\
      &\le \sqrt{ \Prob(A^c) } \left(\sqrt{ \E[D_p(o,x)^2]} + \sqrt{\E[\tilde{D}_p(o,x)^2]}\right).
   \end{align*}
   Since $\sqrt{ \Prob(A^c) } = o((\log d_G(o,x))^{-C})$, it then suffices to show that
   $ \E[D_p(o,x)^2]$ and $\E[\tilde{D}_p(o,x)^2]$ are $O(d_G(o,x)^2 (\log d_G(o,x))^{2C})$.
   
   We have a deterministic bound $\tilde{D}_p(o,x)^2 = O( d_G(o,x)^2 (\log d_G(o,x))^{2C})$,
   so it remains to consider $D_p(o,x)^2$.
   Let $K, c$ be as in Theorem \ref{thm:main}. Then, we have
   \[
      \E[ D_p(o,x)^2 ] = \E[ d_{G_p}(\mathring{o}, \mathring{x})^2 ] 
      \le 4K^2 d_G(o,x)^2 + \sum_{t \ge 2 d_G(o,x) } K^2 (t+1)^2 \Prob(K(t+1)\ge d_{G_p}(\mathring{o}, \mathring{x}) \ge Kt ).% \\
      %&\le 4K^2 d_G(o,x)^2 + \sum_{t \ge 2K d_G(o,x) } t^2 \left( \Prob( d_G(o, \mathring{o}) > t/2) + \Prob( d_G(x, \mathring{x}) > t/2)
      %+ \sum_{p \in B(o,t/2), q \in B(x,t/2)} \Prob( \mathring{o}=p, \mathring{x}=q, d_{G_p}(u,v) \ge t ) \right)
   \]
   Using Theorem \ref{thm:main} and Proposition \ref{prop:smallholes}, we then obtain
   \begin{align*}
      \Prob( d_{G_p}(\mathring{o}, \mathring{x}) \ge Kt ) &\le
      \Prob( d_G(o, \mathring{o}) > t/2) + \Prob( d_G(x, \mathring{x}) > t/2)
      + \sum_{\substack{u \in B(o,t/2), \\ v \in B(x,t/2)}} \Prob( \mathring{o}=u, \mathring{x}=v, d_{G_p}(u,v) \ge Kt ) \\
      &\le 2\Prob( d_G(o, \mathring{o}) > t/2)
      + \sum_{\substack{u \in B(o,t/2), \\ v \in B(x,t/2)}} \Prob( u \connects v, d_{G_p}(u,v) \ge Kt ) \\
      &\le O( \sqrt{t} \exp(-\sqrt{5t}) ) + O(t^{2d}) \exp(-c t).
   \end{align*}
   Thus we have
   \[
      \E[D_p(o,x)^2] \le 4K^2 d_G(o,x)^2 + \sum_{t \ge 2 d_G(o,x)} t^2 O( \sqrt{t} \exp( - \sqrt{5 t} ) = O(d_G(o,x)^2),
   \]
   as desired.
\end{proof}

From here we can prove Proposition \ref{prop:goodapprox}.
\begin{proof}[Proof of Proposition \ref{prop:goodapprox}]
   For succinctness of notation, let us denote $M := (\log d_G(o,x))^C$.
   Again (fixing $p>p_c$) let $K$ be as in Theorem \ref{thm:main}.
   Let us also denote by $E(\tilde{D})$ the (finite) set of edges that determine
   $\tilde{D}_p(o,x)$; as remarked above this is contained in $N(\{o,x\},  M d_G(o,x))$.
   Denote by $V(\tilde{D})$ this set of vertices which are endpoints of edges in $E(\tilde{D})$.
   We define the following event:
   \begin{align*}
      A := &\bigcap_{u,v \in V(\tilde{D})} \left( \{u \connectsthru{G_p} v\}^c \cup 
      \left\{ d_{G_p}(u,v) \le K \max\left( d_G(u,v), M \right) \right\} \right) \\
      &\cap \bigcap_{w \in V(\tilde{D})} \left( \{ w \connectsthru{G_p} \infty \} \cup
      \{|C(w)| \le d_G(o,x)/2 \} \right) \\
      &\cap \left\{ d_G(o, \mathring{o}), d_G(x, \mathring{x}) \le M \right\}.
   \end{align*}
   In words, $A$ is the event that in $V(\tilde{D})$, all pairs of points which are connected in $G_p$ have chemical distance
   not too large; and there are no large components except the infinite component; and that
   both $o$ and $x$ are not too far from the infinite component.
   By Theorem \ref{thm:main}, Theorem \ref{thm:nofinitegiant}, and Proposition \ref{prop:smallholes}, we have that
   \begin{align*}
      \Prob(A^c) \le &O\left( \left( M d_G(o,x) \right)^{2d} \right) \exp\left(-cM\right) \\
      &+ O( (M d_G(o,x))^d ) \exp\left(-c'\left(\frac{d_G(o,x)}{2}\right)^{\frac{d-1}{d}}\right) \\
      &+ O\left( \sqrt{M} \exp(-\sqrt{10M}) \right) \\
      &= O(M^{-2}),
   \end{align*}
   so that $A$ satisfies the hypothesis of Proposition \ref{prop:wlog}. Thus, if 
   we show a uniform almost sure bound
   \[
      \ind_A |\tilde{D}_p(o,x) - D_p(o,x)| = o(d_G(o,x)),
   \]
   Proposition \ref{prop:goodapprox} will follow from this and Proposition \ref{prop:wlog}.
   
   So assume that $A$ holds. Denote by $\tilde{o}$ and $\tilde{x}$ the vertices realizing $\tilde{D}_p(o,x)$;
   that is, they satisfy
   \[
      \tilde{D}_p(o,x) = d_{G_p}(\tilde{o},\tilde{x}) + M(d_G(o,\tilde{o}) + d_G(x,\tilde{x})).
   \]
   Recall also that by definition we have $D_p(o,x) := d_{G_p}(\mathring{o}, \mathring{x})$.
   
   Now since $\mathring{o} \connectsthru{G_p} \mathring{x}$, the definition of $\tilde{D}_p(o,x)$
   gives
   \[
      \tilde{D}_p(o,x) = d_{G_p}(\tilde{o},\tilde{x}) + M[d_G(o,\tilde{o}) + d_G(x,\tilde{x})]
      \le d_{G_p}(\mathring{o}, \mathring{x}) + M[d_G(o,\mathring{o}) + d_G(x,\mathring{x})],
   \]
   which rearranges to
   %\begin{equation} \label{eq:tilde-vs-ring}
   %   0 \le M[ (d_G(o,\tilde{o}) - d_G(o,\mathring{o})) + (d_G(x,\tilde{x}) - d_G(x,\mathring{x})) ]
   %   \le d_{G_p}(\mathring{o},\mathring{x}) - d_{G_p}(\tilde{o}, \tilde{x}),
   %\end{equation}
   %where the left inequality comes from the definition of $\mathring{o}, \mathring{x}$.
   \begin{align}
      d_G(o,\tilde{o}) + d_G(x, \tilde{x}) 
      &\le \frac{1}{M}[d_{G_p}(\mathring{o},\mathring{x}) - d_{G_p}(\tilde{o},\tilde{x})] 
      + d_G(o,\mathring{o}) + d_G(x,\mathring{x}) \nonumber \\
      &\le \frac{1}{M}[d_{G_p}(\mathring{o},\mathring{x}) - d_{G_p}(\tilde{o},\tilde{x})] + 2M, \label{eq:rearrangement}
   \end{align}
   where in the second line we use the bound on $d_G(o, \mathring{o}), d_G(x, \mathring{x})$ guaranteed
    by $A$.
    
   We see from \eqref{eq:rearrangement} that an upper bound on $d_{G_p}(\mathring{o}, \mathring{x}) - d_{G_p}(\tilde{o},\tilde{x})$
   will easily yield an upper bound on 
   $|\tilde{D}_p(o,x) - D_p(o,x)| = |d_{G_p}(\tilde{o},\tilde{x}) - d_{G_p}(\mathring{o}, \mathring{x})
   + M[d_G(o,\tilde{o}) + d_G(x,\tilde{x})]|$.
   %Arguing that $|\tilde{D}_p(o,x) - D_p(o,x)| = |d_{G_p}(\tilde{o},\tilde{x}) - d_{G_p}(\mathring{o}, \mathring{x})
   %+ M[d_G(o,\tilde{o}) + d_G(x,\tilde{x})]|$ is small will essentially come down to showing that 
   %the quantities in \eqref{eq:tilde-vs-ring} are small. 
   To this end, note that by the triangle inequality,
   %we can further bound the right-hand side of \eqref{eq:tilde-vs-ring} by
   \begin{equation} \label{eq:further-upper-bound}
      d_{G_p}(\mathring{o},\mathring{x}) - d_{G_p}(\tilde{o}, \tilde{x}) \le 
      d_{G_p}(\mathring{o}, \tilde{o}) + d_{G_p}(\mathring{x}, \tilde{x});
   \end{equation}
   but this bound is not useful unless we know that $\mathring{o} \connects \tilde{o}$
   and $\mathring{x} \connects \tilde{x}$. To argue this that these are connected, we only need argue
   that $\tilde{o}$ and $\tilde{x}$ lie in the infinite component, 
   since $\mathring{o}$ and $\mathring{x}$ lie in the infinite component by definition.
   Since on $A$ there are no large finite components intersecting $V(\tilde{D})$,
   it then suffices to show that $\tilde{o}$ and $\tilde{x}$ lie in a component of
   size at least $d_G(o,x)/2$, and to do \emph{this} it suffices
   to show that $d_G(\tilde{o}, \tilde{x}) \ge d_G(o,x)/2$,
   since $\tilde{o}$ and $\tilde{x}$ are connected in $G_p$ by definition.
   
   To show this lower bound, first note that by the triangle inequality
   \begin{equation} \label{eq:first-lower-bound}
      d_G(\tilde{o}, \tilde{x}) \ge d_G(o,x) - [d_G(o, \tilde{o}) + d_G(x, \tilde{x})].
   \end{equation}
   Moreover, combining \eqref{eq:rearrangement} with the trivial bound $d_{G_p}(\tilde{o},\tilde{x}) \ge 0$
   gives the upper bound in the first line below,
   which we continue using the assumption that $A$ holds:
   \begin{align*}
      d_G(o,\tilde{o}) + d_G(x, \tilde{x})  
      &\le \frac{1}{M}d_{G_p}(\mathring{o}, \mathring{x}) +2M  \\
      &\le \frac{K}{M}d_G(\mathring{o}, \mathring{x}) + 2M \\
      &\le \frac{K}{M}[d_G(o,x) + d_G(o, \mathring{o}) + d_G(x, \mathring{x})] + 2M \\
      &\le \frac{K}{M}d_G(o,x) + 2K + 2M.  
   \end{align*}
   %The inequality on the second line comes from the trivial bound $d_{G_p}(\tilde{o},\tilde{x}) \ge 0$
   %and the condition $d_G(o,\mathring{o}), d_G(x, \mathring{x}) \le \frac{M}{8K}$ guaranteed by $A$.
   The inequality on the second line follows from the chemical distance condition given by $A$
   (as well as the fact that $\mathring{o}, \mathring{x} \in V(\tilde{D})$
   and $d_G(\mathring{o}, \mathring{x}) \ge M$, which are both guaranteed
   by the fact that $d_G(o, \mathring{o}), d_G(x, \mathring{x})$ are small when $d_G(o,x)$ is
   sufficiently large).
   The third line is again the triangle inequality, and the fourth line comes from the
   upper bound on $d_G(o, \mathring{o}), d_G(x, \mathring{x})$.
   
   Combing the above with \eqref{eq:first-lower-bound} then gives
   \begin{align*}
      d_G(\tilde{o}, \tilde{x}) &\ge d_G(o,x) - \left[ \frac{K}{M}d_G(o,x) + 2K + 2M\right] \\
      &= \left(1 - \frac{K}{M}\right) d_G(o,x) - 2K - 2M \\
      &\ge \frac{1}{2} d_G(o,x),
   \end{align*}
   where the last line holds as long as $d_G(o,x)$ is sufficiently large. Thus, $\tilde{o}$ and $\tilde{x}$
   lie in a component of size at least $d_G(o,x)/2$, hence they lie in the infinite component (since $A$ holds),
   and therefore $\tilde{o}, \tilde{x}, \mathring{o},$ and $\mathring{x}$ are all connected in $G_p$, 
   as desired.
   
   Therefore, we can use the chemical distance condition guaranteed by $A$ to continue \eqref{eq:further-upper-bound}:
   \begin{align}
      d_{G_p}(\mathring{o},\mathring{x}) - d_{G_p}(\tilde{o}, \tilde{x})
      &\le K \left[ \max\left( d_G(\mathring{o}, \tilde{o}), M \right) 
                      + \max\left( d_G(\mathring{x}, \tilde{x}), M \right)\right] \nonumber \\
      &\le K[ d_G(o, \mathring{o}) + d_G(o, \tilde{o}) + d_G(x, \mathring{x}) + d_G(x, \tilde{x}) ] + 2M \nonumber \\
      &\le K[ d_G(o, \tilde{o}) + d_G(x, \tilde{x}) ] + 4M, \label{eq:top-upper-bound}
   \end{align}
   where the last line comes again from the upper bounds on $d_G(o,\mathring{o}), d_G(x,\mathring{x})$ given by $A$.
   
   Recombining \eqref{eq:top-upper-bound} with \eqref{eq:rearrangement} then gives
   \begin{align*}
      d_G(o, \tilde{o}) + d_G(x,\tilde{x}) &\le
      \frac{1}{M}[d_{G_p}(\mathring{o}, \mathring{x}) - d_{G_p}(\tilde{o}, \tilde{x})] + 2M \\
      &\le \frac{K}{M}\left[ d_G(o, \tilde{o}) + d_G(x, \tilde{x}) \right] + 4K + 2M
   \end{align*}
   which rearranges to
   \begin{align}
      d_G(o, \tilde{o}) + d_G(x, \tilde{x}) &\le \left( 1 - \frac{K}{M} \right)^{-1} \left( 4K + 2M \right) \nonumber \\
      &\le 8K + 4M, \label{eq:aux-upper-bound}
   \end{align}
   where the last line holds whenever $d_G(o,x)$ is sufficiently large that $K/M < 1/2$.
   Plugging \eqref{eq:aux-upper-bound} back into \eqref{eq:top-upper-bound} further gives
   \begin{equation} \label{eq:top-upper-bound-final}
      d_{G_p}(\mathring{o}, \mathring{x}) - d_{G_p}(\tilde{o}, \tilde{x}) 
      \le 8K^2 + 8KM.
   \end{equation}
   Finally, putting together \eqref{eq:aux-upper-bound} and \eqref{eq:top-upper-bound-final} then gives 
   that on $A$,
   \begin{align*}
      |D_p(o,x) - \tilde{D}_p(o,x)| &= 
      |d_{G_p}(\mathring{o},\mathring{x}) - d_{G_p}(\tilde{o}, \tilde{x}) - M[d_G(o, \tilde{o}) + d_G(x, \tilde{x})]| \\
      &\le d_{G_p}(\mathring{o},\mathring{x}) - d_{G_p}(\tilde{o}, \tilde{x}) + M[d_G(o, \tilde{o}) + d_G(x, \tilde{x})] \\
      &\le (8K^2+KM) + (8KM + 4M^2) = O(M^2) = o(d_G(o,x)),
   \end{align*}
   as desired. 
\end{proof}

\subsection{ Bounding the derivative of $\E \tilde{D}_p$ by a sum over the realizing geodesic}
Let $\pi$ be a (random) open path realizing $\tilde{D}(o,x)$; that is, $\pi$ is
an open path between two vertices $\tilde{o}$ and $\tilde{x}$ such that
\begin{equation} \label{eq:dtilde}
   \tilde{D}_p(o,x) = |\pi| + M(d_G(o,\tilde{o}) + d_G(x,\tilde{x}))
\end{equation}
(where again $M = d_G(o,x)^C$).
The goal of this subsection is to prove the following lemma:

\begin{lemma} \label{lem:bound_by_sum}
   Fix $p_0 > p_c(G)$. Then for all $p \in [p_0, 1)$ we have
\[   \left| \frac{d}{dp} \E[\tilde{D}_p(o,x)] \right| \le
   \frac{1}{p} \E\left[ \sum_{e =\{u,v\} \in \pi} 
   \ind_{\left\{u \connectsthru{G_p \setminus \{e\}} v\right\}} d_{G_p \setminus \{e\}}(u,v) \right]
   + \frac{1}{1-p} o(d_G(o,x)),
\]
where the implicit constant in the little o notation may depend on $p_0$ and $G$ but not on $p$.
\end{lemma}

The first ingredient in the proof of Lemma \ref{lem:bound_by_sum} is Russo's formula,
which we can apply since
$\tilde{D}_p$ is a decreasing random variable depending only on a finite set of edges. Recall:
\begin{prop}(Russo's Formula; see e.g. \cite[Theorem 2.32]{grimmett1999percolation})
   Let $f : 2^E \to \R$ be a decreasing function, in the sense that $\omega \subset \omega'$ implies
   $f(\omega) \ge f(\omega')$. Suppose further that $f$ only depends on a finite set of edges
   (that is, there exists $F \subset E$ finite such that $f(\omega) = f(\omega \cap F)$ for all $\omega \in 2^E$).
   Then $\E[f(G_p)]$ is differentiable in $p$, and its derivative is given by
   \[ 
      \frac{d}{dp} \E[f(G_p)] = - \sum_{e \in E} \E[ \Delta_e f (G_p) ],
   \]
   where for $\omega \in 2^E$ and $e \in E$ we define 
   $ \Delta_e f(\omega) := f(\omega \setminus \{e\}) - f(\omega \cup \{e\})$.
\end{prop}
%\textcolor{red}{If I can't find a citation for this precise version here's a proof
%\begin{proof}[Just to triple check]
%   Replacing $f$ by $-f$, we can just prove the increasing version (for simplification of signs).
%   \begin{align*}
%      \E[f(G_{p+\epsilon}) - f(G_p)] &= \E[\ind_{\{|G_{p+\epsilon} \setminus G_p| = 1\}}[f(G_{p+\epsilon}) - f(G_p)]] + O(\epsilon^2) \\
%      &= \E\left[ \sum_{e \in F} \ind_{\{G_{p+\epsilon} = G_p \sqcup \{e\}\}} \Delta_e f(G_p)\right] + O(\epsilon^2) \\
%      &= \sum_{e \in F} \E[ \ind_{\{ e \in G_{p+\epsilon}, e \notin G_p \}} \Delta_e f(G_p) ] + O(\epsilon^2) \\
%      &= \sum_{e \in F} \epsilon \E[ \Delta_e f(G_p) ] + O(\epsilon^2),
%   \end{align*}
%   so we're done!
%\end{proof}
%}

The other ingredient is a geometric condition that holds with very high probability.
This condition will allow us to assume that the endpoints of most edges $e$ in the realizing
geodesic $\pi$ are still connected in $G_p \setminus \{e\}$, which will be helpful
for bounding the influence $\Delta_e \tilde{D}(o,x)$.

More specifically, define the following ``deletion-tolerant'' version of the uniqueness event $A_v(M)$
(where here we again take $M := (\log d_G(o,x))^C$):
\begin{equation} \label{eq:deleted-uniqueness}
   A'_{v,e} := 
   %\left \{ B(v, \frac{R}{10}) \connectsthru{G_p \setminus \{e\}} \bnd B(v,R) \right\} \cap
   \left\{ \begin{array}{c}
       \mbox{ at most one component of } (G_p \setminus \{e\}) \cap B(v,R) \\
   \mbox{ intersects both } B(v, \frac{M}{5}) \mbox{ and } \bnd B(v, \frac{M}{2}) 
   \end{array} \right\}.
\end{equation}
Then for each edge $e$ of $E(\tilde{D})$ (recall that $E(\tilde{D})$ is the finite
set of edges which $\tilde{D}(o,x)$ depends on), arbitrarily choose one endpoint $v(e)$ of $e$, and then define
\[
   A' := \bigcap_{e \in E(\tilde{D})} A'_{v(e),e}.
\]
Then $A'$ holds with very high probability:
\begin{prop}
   Let $p_0 > p_c$. Then there exists $R_0 < \infty$ such that whenever $M > R_0$, we have
   \[
      \Prob_p((A')^c) \le \frac{1}{1-p} O( [Md_G(o,x)]^d ) \exp(- \sqrt{M})
   \]
   for all $p \in [p_0,1]$, where the implicit constants only depend on the graph $G$, not $p$ or $x$.
\end{prop}
\begin{proof}
   This is essentially a ``finite-energy argument''.
   Note that $A'_{v,e}$ is independent of $e$, and therefore
   \begin{align*}
      (1-p)\Prob_p({A'_{v,e}}^c) &= \Prob_p({A'_{v,e}}^c, e \mbox{ closed }) \\
      &\le \Prob\left( \left\{ \begin{array}{c}
          \mbox{ at most one component of } G_p \cap B(v,R) \\
          \mbox{ intersects both } B(v, \frac{M}{5}) \mbox{ and } \bnd B(v, \frac{M}{2}) 
      \end{array} \right\}^c \right) \\
      &\le \exp(-\sqrt{M}).
   \end{align*}
   Therefore a union bound gives
   \begin{align*}
      \Prob({A'}^c) \le \sum_{e \in E(\tilde{D})} \Prob({A_{v(e),e}}^c) \le \frac{1}{1-p} O([Md_G(o,x)]^d)\exp(-\sqrt{M}).
   \end{align*}
\end{proof}

%Now we take $f(G_p) = \tilde{D}_p(o,x)$. In order to show continuity of $\lim D_p(o,x)/d_G(o,x)$,
%we want to get a bound of the form $\frac{d}{dp} \tilde{D}_p(o,x) \le C(p)d_G(o,x)$, where 
%$C(p)$ is continuous in $p$.

\begin{proof}[Proof of Lemma \ref{lem:bound_by_sum}]
To apply Russo's formula, take $f(G_p) = \tilde{D}(o,x)$.
Note that since $\Delta_e f(G_p)$ is independent of whether the edge $e$ is open or closed,
\begin{align*}
   \frac{d}{dp} \E[f(G_p)] &= - \sum_{e \in E} \frac{1}{p}\E[ \ind_{\{e \mbox{ open}\}} \Delta_e f (G_p) ] \\
   &= - \frac{1}{p} \E\left[ \sum_{e \in E} \ind_{\{e \mbox{ open}\}} \Delta_e f (G_p) \right].
\end{align*}

%\textcolor{red}{Recall} Next, define $\pi$ to be the a (random) open path realizing $\tilde{D}(o,x)$; that is, $\pi$ is
%an open path between two vertices $\tilde{o}$ and $\tilde{x}$ such that
%\begin{equation}\label{eq:dtilde}
%   \tilde{D}_p(o,x) = |\pi| + M(d_G(o,\tilde{o}) + d_G(x,\tilde{x}).
%\end{equation}
Note that if $e \notin \pi$, we have $\ind_{\{e \mbox{ open}\}} \Delta_e f(G_p) = 0$, since in that case
closing $e$ cannot increase $\tilde{D}_p(o,x)$. Thus we have
\begin{equation} \label{eq:just_on_geo}
   \frac{d}{dp} \E[f(G_p)] = - \frac{1}{p} \E\left[ \sum_{e \in \pi} \ind_{\{e \mbox{ open}\}} \Delta_e f (G_p) \right].
\end{equation}

Next, note that on the event $A'$, for any edge $e = \{u,v\} \in \pi \setminus N( \{\tilde{o}, \tilde{x}\}, M)$, we have
that $u \connectsthru{G_p \setminus \{e\}} v$. This is because in that case, the segments of $\pi$ before 
and after $e$ connect $B(v,M/5)$ to $\bnd B(v,M/2)$, and hence are part of the same component.
Therefore for such $e$, $\Delta_e f(G_p) \le \ind_{\left\{u \connectsthru{G_p \setminus \{e\}} v\right\}} d_{G_p \setminus \{e\}}(u,v)$.

         \begin{figure}
   % !TEX root =chemical distance polynomial growth.tex

\tikzset{every picture/.style={line width=0.75pt}} %set default line width to 0.75pt        

\begin{tikzpicture}[x=0.75pt,y=0.75pt,yscale=-1,xscale=1]

%Shape: Circle [id:dp13167020449395983] 
\draw  [color={rgb, 255:red, 74; green, 144; blue, 226 }  ,draw opacity=1 ] (67,130) .. controls (67,116.19) and (78.19,105) .. (92,105) .. controls (105.81,105) and (117,116.19) .. (117,130) .. controls (117,143.81) and (105.81,155) .. (92,155) .. controls (78.19,155) and (67,143.81) .. (67,130) -- cycle ;
%Shape: Circle [id:dp38010474388333004] 
\draw  [color={rgb, 255:red, 74; green, 144; blue, 226 }  ,draw opacity=1 ] (402,177) .. controls (402,163.19) and (413.19,152) .. (427,152) .. controls (440.81,152) and (452,163.19) .. (452,177) .. controls (452,190.81) and (440.81,202) .. (427,202) .. controls (413.19,202) and (402,190.81) .. (402,177) -- cycle ;
%Shape: Free Drawing [id:dp30540516389522343] 
\draw  [line width=1.5] [line join = round][line cap = round] (94.2,129.6) .. controls (100,129.6) and (106.02,126.33) .. (111.2,124.6) .. controls (111.65,124.45) and (111.99,123.18) .. (112.2,123.6) .. controls (113.47,126.14) and (112.8,130.4) .. (115.2,131.6) .. controls (119.05,133.52) and (131.87,132.62) .. (132.2,132.6) .. controls (141.09,132.04) and (157.18,115.58) .. (166.2,124.6) .. controls (167.58,125.98) and (178.1,133.88) .. (178.2,135.6) .. controls (178.53,141.26) and (178.76,146.96) .. (178.2,152.6) .. controls (177.69,157.68) and (172.68,160.53) .. (174.2,166.6) .. controls (175.07,170.1) and (183.12,171.57) .. (186.2,172.6) .. controls (186.2,172.6) and (189.62,175.43) .. (190.2,175.6) .. controls (202.75,179.19) and (216.36,179.18) .. (227.2,184.6) .. controls (231,186.5) and (229.86,201.17) .. (230.2,205.6) .. controls (230.94,215.18) and (234.6,221.5) .. (240.2,230.6) .. controls (241.1,232.06) and (244.84,241.91) .. (249.2,242.6) .. controls (261.45,244.53) and (283.03,238.77) .. (292.2,229.6) .. controls (299.91,221.89) and (303.42,205.46) .. (317.2,204.6) .. controls (330.76,203.75) and (342.49,207.89) .. (351.2,216.6) .. controls (363.19,228.59) and (361.64,238.29) .. (380.2,236.6) .. controls (385.77,236.09) and (390.81,227.99) .. (394.2,224.6) .. controls (396.29,222.51) and (396.45,215.35) .. (399.2,212.6) .. controls (405.63,206.17) and (409.06,196.87) .. (413.2,188.6) .. controls (414.42,186.17) and (419.34,184.46) .. (421.2,182.6) .. controls (422.61,181.19) and (425.85,176.6) .. (428.2,176.6) ;
%Curve Lines [id:da08214994166103518] 
\draw [color={rgb, 255:red, 189; green, 16; blue, 224 }  ,draw opacity=1 ][fill={rgb, 255:red, 189; green, 16; blue, 224 }  ,fill opacity=1 ][line width=1.5] [line join = round][line cap = round]   (193.2,176.4) .. controls (192.38,176.13) and (199.38,178.13) .. (203.2,178.4) .. controls (205.38,179.13) and (215.37,180.12) .. (217.2,182.4) .. controls (218.57,196.09) and (207.37,215.81) .. (198.2,220.4) .. controls (194.05,222.48) and (192.41,226.47) .. (186.2,224.4) .. controls (185.41,224.14) and (185.34,214.02) .. (185.2,213.4) .. controls (184.79,211.54) and (170.83,211.56) .. (170.2,208.4) .. controls (168.52,199.99) and (180.55,192.05) .. (184.2,188.4) .. controls (188.82,183.78) and (186.32,176.4) .. (193.2,176.4) -- cycle ;

\draw  [line width=2,red] (200,176)--(207,177);
% Text Node
\draw (200,160) node [anchor=north west][inner sep=0.75pt]   [align=left] {$e=\{p,q\}$};
\draw (83,118) node [anchor=north west][inner sep=0.75pt]   [align=left] {$\tilde{o}$};
\draw (92,130) node[circle,fill=black,inner sep=1.2pt] {} ;
\draw (428.2,176.6) node [anchor=north west][inner sep=0.75pt]   [align=left] {$\tilde{x}$};
\draw (428.2,176.6) node[circle,fill=black,inner sep=1.2pt] {} ;

\draw (110,170) node [anchor=north west][inner sep=0.75pt]   [align=left] {${o}$};
\draw (122,170) node[circle,fill=black,inner sep=1.2pt] {} ;
\draw (450,146.6) node [anchor=north west][inner sep=0.75pt]   [align=left] {${x}$};
\draw (455,142) node[circle,fill=black,inner sep=1.2pt] {} ;

\end{tikzpicture}
   
\caption{Away from $\tilde{o}$ and $\tilde{x}$ each edge $e$ has a bypass,
and we can bound its influence $\Delta_e f$ by the length of the bypass. \label{fig:cost}}
   \end{figure}
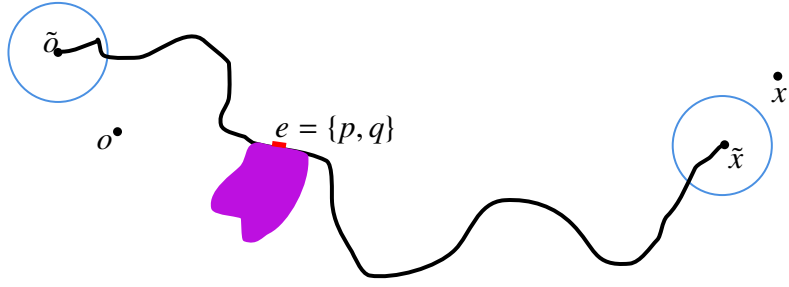
    %%%%%%%%%%%%%%%%%

Moreover, any edge of distance at most $M$ from $\tilde{o}$ or $\tilde{x}$ can contribute
at most $M^2$ to $\tilde{D}_p(o,x)$ when it is closed, since if $e \in \pi \cap B(\tilde{o}, M)$,
taking $\pi'$ to be the connected component of $\pi$ containing $\tilde{x}$ gives an open
path from some $\tilde{o}'$ to $\tilde{x}$ with $d_G(o, \tilde{o}') \le d_G(o, \tilde{o}) + M$.
(The argument for edges near $\tilde{x}$ is similar.)
Therefore, we have
\begin{align*}
   \E\left[ \ind_{A'} \sum_{e \in \pi} \ind_{\{e \mbox{ open}\}} \Delta_e f (G_p) \right] 
   &\le M^2 O(M^d) + \E\left[ \sum_{e =\{u,v\} \in \pi} 
   \ind_{\left\{u \connectsthru{G_p \setminus \{e\}} v\right\}} d_{G_p \setminus \{e\}}(u,v) \right] \\
   &= \E\left[ \sum_{e =\{u,v\} \in \pi} 
   \ind_{\left\{u \connectsthru{G_p \setminus \{e\}} v\right\}} d_{G_p \setminus \{e\}}(u,v) \right]
   + o(d_G(o,x)).
\end{align*}

Lastly, if the event $A$ does not hold, we can use the following crude bound.
Since $0 \le \tilde{D}_p(o,x) \le M d_G(o,x)$ almost surely, for every edge $e \in E(\tilde{D})$, $\Delta_e f(G_p) \le M d_G(o,x)$.
Therefore
\begin{align*}
   \E\left[ \ind_{(A')^c} \sum_{e \in \pi} \ind_{\{e \mbox{ open}\}} \Delta_e f (G_p) \right] 
   &\le \Prob_p(A'^c) (M d_G(o,x)) |E(\tilde{D})| \\
   &= \frac{1}{1-p} o(d_G(o,x)),
\end{align*}
where we have used that our bound on $(1-p)\Prob_p(A'^c)$ decays faster than any inverse polynomial of $d_G(o,x)$.
Combining these two bounds with \eqref{eq:just_on_geo} then gives the desired bound.
\end{proof}

\subsection{Greedy lattice animal bounds}
From here we want to show that the expectation
in the right hand side of Lemma \ref{lem:bound_by_sum} is $O(d_G(o,x)$ (with implicit constant independent of 
$p \in [p_0,1]$ and $x$). 
Lemma \ref{lem:linear-bound} below will allow us to conclude this once we show that (roughly speaking):
first, the quantities $\ind_{\left\{u \connectsthru{G_p \setminus \{e\}} v\right\}} d_{G_p \setminus \{e\}}(u,v)$
that we sum have good tails; 
second, although not independent, these random variables have a weak form of independence;
lastly, the path $\pi$ that we sum over is $O(d_G(o,x))$ with very 
high probability.

Let us begin by arguing that the path we sum over is not too long.
For technical reasons, it will be easier for us to control sums over random paths starting at $o$
rather than a random starting point. Therefore, let us define $\bar{\pi}$
to be the concatenation of the following paths: first, an edge-geodesic from $o$
to $\tilde{o}$; then $\pi$; then an edge-geodesic from $\tilde{x}$ to $x$. Note that
   \[
      |\bar{\pi}| := d_G(o, \tilde{o}) + |\pi| + d_G(x, \tilde{x}).
   \]
We have the following control on the length of $|\bar{\pi}|$:
\begin{prop} \label{prop:pi-bar-bound}
   Then there exists $K' < \infty$ such that whenever $d_G(o,x)$ is sufficiently large,
   for any $t \ge d_G(o,x)$, we have
   \[
      \Prob_p( |\bar{\pi}| \ge K't) \le \exp(-t^{1/5}).
   \]
   for all $p \in [p_0,1]$.
\end{prop}
\begin{proof}
   We have that almost surely by \eqref{eq:dtilde}
   \[
      |\bar{\pi}| \le \tilde{D}_p(o,x) \le Md_G(o,\mathring{o}) + d_{G_p}(\mathring{o}, \mathring{x}) + Md_G(\mathring{x},x),
   \]
   where recall that $\mathring{o}$ and $\mathring{x}$ are the $d_G$-closest points of
   the infinite component of $G_p$ to $o$ and $x$ respectively. Therefore it suffices to find $C< \infty, c>0$ such that
   \[
      \Prob_p(Md_G(o,\mathring{o}) + d_{G_p}(\mathring{o}, \mathring{x}) + Md_G(\mathring{x},x)  \ge K't) \le 
      \exp(-t^{1/5}).
   \]
   We have 
   \begin{align*}
      \Prob_p(Md_G(o,\mathring{o}) &+ d_{G_p}(\mathring{o}, \mathring{x}) + Md_G(\mathring{x},x) \ge K' t) \le \\
      &\Prob_p(d_G(o,\mathring{o}) > \sqrt{t}) + \Prob_p(d_G(x,\mathring{x}) > \sqrt{t}) \\
       &+ \sum_{\substack{u \in B(o,\sqrt{t}) \\ v \in B(x,\sqrt{t})}}
       \Prob_p\left(\mathring{o}=u,\mathring{x}=v, d_{G_p}(u,v) > K't - 2M\sqrt{t} \right).
   \end{align*}
   So given $p_0 > p_c$, fix $K$ as in Theorem \ref{thm:main}, and fix any $K' > K$.
   Then whenever $d_G(o,x)$ is sufficiently large (depending on $K$ and $K'$),
   for any $t \ge d_G(o,x)$ we have $K't - 2M\sqrt{t} > K(t - 2\sqrt{t})$
   and $t - 2\sqrt{t} \ge d_G(u,v)$ for any $u \in B(o,\sqrt{t}), v \in B(o,\sqrt{t})$.
   Thus by Theorem \ref{thm:main} and Proposition \ref{prop:smallholes}, the above is bounded by
   \begin{align*}
      O(t^{1/4})\exp(-\sqrt{10}t^{1/4})) + O(\sqrt{t}^{2d})\exp(-(t-\sqrt{t})),
   \end{align*}
   (with all implicit constants only depending on $G, K,$ and $K'$, not on $p$ or $t$).
   Whenever $d_G(o,x)$ is sufficiently large, for all $t \ge d_G(o,x)$,
   the above is smaller than $\exp(-t^{1/5})$.
\end{proof}

Now, let us show good tails for the random variables $\ind_{\left\{u \connectsthru{G_p \setminus \{e\}} v\right\}} d_{G_p \setminus \{e\}}(u,v) $:
\begin{lemma} \label{lem:deleted-tails}
   Given $p_0 > p_c$, there exists $c > 0$ such that we have
   \[
      \Prob_p\left( \ind_{\left\{u \connectsthru{G_p \setminus \{e\}} v\right\}} d_{G_p \setminus \{e\}}(u,v) \ge t \right)
      \le \exp(-ct).
   \]
   for all edges $e = \{u,v\}$, $t \ge 1$, and $p \in [p_0, 1]$.
\end{lemma}
\begin{proof}
   Note that if we replaced $G_p \setminus \{e\}$ with $G_p$, this lemma would be a
   the special case of Theorem \ref{thm:main} where $d_G(u,v) = 1$;
   the proof of this lemma will be quite similar.
   (In fact, Theorem \ref{thm:main} and
   a finite energy argument yields this statement up to a factor of $1/(1-p)$; here we will
   do just a little more work in order to get Lipschitz continuity).
   
   We again do the same coarse-graining $\hat{G}$ of $G$, here taking every $V^R$ to include $u$ so that
   $\hat{u}=\hat{v}$. We then define $F$ in the same way (and assume that $K' \ge 2$).
   Note that in this case $F$ is the union of the closed
   clusters of all the sites in $\hat{G}$ intersecting $B_{\hat{G}}(\hat{\rho}(u), K')$.
   
   We then claim that we have a version of Lemma \ref{lem:geolem} when $e$
   is forced to be closed. That is, we claim that if $u \connectsthru{G_p \setminus \{e\}} v$,
   then there exists a path in $G_p \setminus \{e\}$ which lies in 
   $\bigcup \{B(w,R) : \hat{w} \in N(\hat{u} \cup F, \Delta)\}$.
   To prove this, we treat $\hat{u}$ as a closed site (even if $A_u$ holds)
   and then run the proof of Lemma \ref{lem:geolem};
   since we treat $\hat{u}$ as a closed site
   none of the surgeries we perform to create our microscopic
   open path can involve $e$, so the claim is shown.
   
   The proof then continues as does the proof of Theorem \ref{thm:main};
   we know that an open path lies in $N_{\hat{G}}(\hat{u} \cup F, \Delta)$,
   and we know that $|F|$ has exponential tails,
   so we get that the random variable in question has exponential tails.
\end{proof}

We now use ``greedy lattice animal'' bounds as in \cite{CNN}
to conclude that our sum of variables with good tails over a likely-at-most-linear
path has linear expectation.
\begin{lemma}[Cf. Lemma 2.7 in \cite{CNN}] \label{lem:linear-bound}
   Let $(X_e)_{e \in E}$ be a family of $\N$-valued random variables indexed by the
   edges of a transitive graph $G$ with polynomial growth of degree $d$.
   Suppose that for each $N \ge 1$ and each $2N$-separated
   subset $S$ of $E$, the family $\{X_e = N\}_{e \in S}$ is independent.
   Define $q_N := \sup_{e \in E} \Prob(X_e = N)$.
   Suppose that for some $B <\infty$
   we have $\sum_{N=0}^{\infty} N^{d+2} q_N \le B$.
   Then there exists $C$ depending only on $G$ and $B$ such that the following holds.
   For any random path $\pi$ starting from $o$ and any $L \in \N$ we have
   \[
      \E\left[ \sum_{e \in \pi} X_e \right] \le CL + C \sum_{\ell \ge L} \ell \Prob( |\pi| = \ell )^{1/2}.
   \]
\end{lemma}
The proof of Lemma \ref{lem:linear-bound} is exactly the same as the proof
of Lemma 2.7 in \cite{CNN}; we simply need the following
analogue of their Lemma 2.6 for graphs of polynomial growth:
\begin{lemma}[Cf. Lemma 2.6 in \cite{CNN}] \label{lem:greedy}
   Let $G$ be a transitive graph which has polynomial growth of degree $d$.
   Given $N \in \N$, let $I_{e,N}$ be a collection of Bernoulli random variables indexed by
   edges $e$ of $G$ such that for every $2N$-separated subset $S$ of $E$,
   the family $\{I_{e,N}\}_{e \in S}$ is independent.
   Define
   \[
      \Gamma_{L,N} := \max \left\{ \sum_{e \in \pi} I_{e,N} : \pi \mbox{ a path from } o \mbox{ of length at most } L \right\},
   \]
   \[
      q_N := \sup_{e \in E} \E[ I_{e,N} ].
   \]
   Then there exists a positive constant $C$ depending on $G$ such that for all $L \in \N$,
   \[
      \E[ \Gamma_{L,N} ] \le CLN^d q_N^{1/d}.
   \]
\end{lemma}
\begin{proof}
   Consider the graph $\tilde{G}^N$ whose vertex set $\tilde{V}$ equal to the edge set $E$ of $G$,
   and whose edges consist of pairs of edges of $G$ which have distance at most $2N$.
   Since $G$ has polynomial growth of degree $d$, there exists some constant $C$ such that
   $\tilde{G}^N$ has maximum degree at most $C(2N)^d$.
   Therefore, $\tilde{G}^N$ can be colored by at most $C(2N)^d + 1$ colors (using a greedy coloring).
   That is, $E$ is a disjoint union of at most $C(2N)^d + 1$ sets $E_i \subset E$, with each $E_i$
   a $2N$-separated set; hence by assumption the family $\{ I_{e,N} \}_{e \in E_i}$ is independent for each $i$.
   
   Then, if we let $\{ \bar{I}_e \}_{e \in E}$ be an i.i.d. family such that each $I_{e,N}$ has the 
   same distribution as $I_{e,N}$, we have
   \[
      \E[ \Gamma_{L,N} ] \le \sum_i \E\left[ \max_{ |\pi| \le L } \sum_{\pi \cap E_i} I_{e,N} \right]
      = \sum_i \E\left[ \max_{ |\pi| \le L } \sum_{e \in \pi \cap E_i} \bar{I}_{e,N} \right]
      \le (C(2N)^d + 1)\E\left[ \max_{ |\pi| \le L } \sum_{e \in \pi} \bar{I}_{e,N} \right].
   \]
   So we will be done once we establish that 
   \[
      \E\left[ \max_{ |\pi| \le L } \sum_{e \in \pi} \bar{I}_{e,N} \right] = O(L q_N^{1/d}),
   \]
   or equivalently that there exists a constant $C(G)$ depending only on the graph $G$ such that
   \[
      \frac{\E\left[ \max_{ |\pi| \le L } \sum_{e \in \pi} \bar{I}_{e,N} \right]}{L q_N^{1/d}} \le C(G).
   \]
   To prove this we use a Peierls argument which closely follows Lemma 6.8 of \cite{DHS}
   but applies to any transitive graph of polynomial growth (or in fact to any graph of ``strict'' polynomial
   growth, see Remark \ref{rem:strict-poly} below).
   
   First, let us consider the case that $L q_N^{1/d} \le 1$. Then
   \begin{align*}
      \frac{\E\left[ \max_{ |\pi| \le L } \sum_{e \in \pi} \bar{I}_{e,N} \right]}{L q_N^{1/d}}
      \le \frac{1}{L q_N^{1/d}} \sum_{e \in E(B(o,L))} \E \bar{I}_{e,N} \le CL^d q_N (L q_N^{1/d})^{-1}
      \le C(L q_N^{1/d})^{d-1} \le C.
   \end{align*}
   Now we treat the case that $L q_N^{1/d} > 1$.
   To this end, recall the construction of the coarse-grained graph $\hat{G}(R)$ constructed above
   in Section \ref{sec:intro} for
   some scale $R \ge 1$. Then note that for any fixed path $\pi$ of length $L$, we have 
   \[
      \sum_{e \in \pi} \bar{I}_{e,N} \le \sum_{e \in \bigcup_{\hat{v} \in \hat{\rho}(\pi)} E(\hat{v})} \bar{I}_{e,N}.
   \]
   The proof of Proposition \ref{prop:separate} shows that there exists some constant $c>0$
   (independent of $\pi$ and $R$) such that that $|\hat{\rho}(\pi)| \le (c/R)|\pi| = c(L/R)$.
   Thus $\hat{\rho}(\pi)$ is a connected (vertex) subset of $\hat{G}(R)$ of length
   at most $c(L/R)$ and furthermore the number of edges in the right hand sum above is
   \[
      \left| \bigcup_{\hat{v} \in \hat{\rho}(\pi)} E(\hat{v}) \right| \le c(L/R) \cdot C(R/30)^d = C' L R^{d-1}
   \]
   since each tile $\hat{v}$ is contained in a ball of radius of radius $R/30$.
   
   Therefore we have, for each $s \ge 0$,
   \begin{align*}
      \Prob\left( \max_{ o \in \pi, |\pi| = L} \sum_{e \in \pi} \bar{I}_{e,N} \ge Lq_N^{1/d} s \right)
      &\le \Prob\left( \max_{\substack{\hat{o} \in \hat{\pi} \subset \hat{V}, \\ |\hat{\pi}| \le c(L/R)}}
      \sum_{e \in \bigcup_{\hat{v} \in \hat{\pi}} E(\hat{v})} \bar{I}_{e,N} \ge Lq_N^{1/d} s \right) \\
      &\le 
      \sum_{\substack{\hat{o} \in \hat{\pi} \subset \hat{V}, \\ |\hat{\pi}| \le c(L/R)}}
      \left( \prod_{e \in \bigcup_{\hat{v} \in \hat{\pi}}} \E( \exp(\hat{I}_{e,N} ) \right)
      \exp(-Lq_N^{1/d} s) \\
      &\le
      (2 \hat{D})^{c(L/R)} (1 - q_N + e q_N)^{C' L R^{d-1}} \exp(-Lq_N^{1/d} s) \\
      &\le 
      \exp\left( (\log (2\hat{D}) )\cdot c(L/R) + (e-1) q_N \cdot (C' L R^{d-1}) - Lq_N^{1/d}s \right).
   \end{align*}
   where the second line is a union bound and Chernoff bound, and the third line comes from the fact
   that the number of connected vertex subsets of length at most $L/R$ containing a fixed vertex
   in a graph of degree at most
   $\hat{D}$ is at most $(2 \hat{D})^{L/R}$---recall from Section \ref{sec:coarse} that there is a constant $\hat{D}$ independent of $R$ such that
   each $\hat{G}(R)$ has degree at most $\hat{D}$.
   The fourth line follows from the fact that $(1 - q_N + eq_N) \le (e^{e-1})^{q_N}$.
   
   Choosing $R = \lceil q_N^{-1/d} \rceil$ then gives us
   \begin{align*}
      \Prob\left( \max_{ o \in \pi, |\pi| = L} \sum_{e \in \pi} \bar{I}_{e,N} \ge Lq_N^{1/d} s \right)
      &\le
      \exp\left( (\log(2\hat{D})) \cdot c (L \lfloor q_N^{1/d} \rfloor)  + 
      (e-1) \cdot C' L q_N \lceil q_N^{-1/d} \rceil^{d-1}
      - Lq_N^{1/d}s  \right) \\
      &\le \exp( L q_N^{1/d}(C'' - s) ).
   \end{align*}
   So by the layer-cake formula we finally obtain
   \begin{align*}
      \E\left[ \frac{\max_{ o \in \pi, |\pi| = L} \sum_{e \in \pi} \bar{I}_{e,N}}{L q_N^{1/d}} \right]
      &\le C'' + \int_{C''}^{\infty} \Prob\left( \frac{\max_{ o \in \pi, |\pi| = L} \sum_{e \in \pi} \bar{I}_{e,N}}{Lq_N^{1/d}} \ge s \right) ds \\
      &\le C'' + \int_{C''}^{\infty} \exp( - L q_N^{1/d}(s - C'') )ds \\
      &\le C'' + \int_{C''}^{\infty} \exp(-(s-C''))ds = C'' + 1. \\
   \end{align*}
   Thus we have proved the desired bound.
\end{proof}
\begin{rmk} \label{rem:strict-poly}
   Even though Lemmas \ref{lem:linear-bound} and \ref{lem:greedy} are stated for \emph{transitive} graphs,
   in fact the proof above shows that the same results hold for any (not necessarily transitive) 
   graph which which has ``strict'' polynomial growth of degree $d$,
   in the sense that there exist $0<c<C<\infty$ such that for all $R \ge 1$ we have
   \[
      cR^d \le \inf_{v \in V} |B(v,R)| \le \sup_{v \in V} |B(v,R) \le CR^d.
   \]
\end{rmk}

\subsection{Proof of Theorem \ref{thm:cty}}
Finally we can prove Theorem \ref{thm:cty}.
\begin{proof}[Proof of Theorem \ref{thm:cty}]
   First we prove the analogous statement with $D_p$ replaced by $\tilde{D}_p$.
   Fix $p_0 > p_c$. First we prove Lipschitz continuity on $[p_0, 1)$.
   For any $p_0 \le p < q < 1$, we have by Lemma \ref{lem:bound_by_sum}
   \begin{align*}
      0 \le \E \tilde{D}_p(o,x) - \E \tilde{D}_q(o,x) &= \int_p^q - \frac{d}{d\rho} \E \tilde{D}_{\rho}(o,x) d\rho \\
      %&= \int_p^q   \frac{1}{\rho} \E\left[ \sum_{e \in \pi} \ind_{\{e \mbox{ open}\}} \Delta_e f (G_p) \right] d\rho \\
      %&\le \frac{1}{p_0} \int_p^q \left( \E\left[ \ind_{A'} \sum_{e \in \pi} \ind_{\{e \mbox{ open}\}} \Delta_e f (G_p) \right] +
      %\E\left[ \ind_{{A'}^c} \sum_{e \in \pi} \ind_{\{e \mbox{ open}\}} \Delta_e f (G_p) \right]
      %\right)d\rho \\
      %&\le
      %O(M^{d+2})(q-p) + \int_p^q 
      &\le \int_p^q \frac{1}{\rho}
      \E\left[ \sum_{e =\{u,v\} \in \pi} \ind_{\left\{u \connectsthru{G_{\rho} \setminus \{e\}} v\right\}} 
      d_{G_{\rho} \setminus \{e\}}(u,v) \right]d\rho \\
      & + \int_p^q \frac{1}{1-\rho} o(d_G(o,x)) d\rho \\
      &\le
      \frac{1}{p_0} \int_p^q 
      \E\left[ \sum_{e =\{u,v\} \in \bar{\pi}} \ind_{\left\{u \connectsthru{G_{\rho} \setminus \{e\}} v\right\}} 
      d_{G_{\rho} \setminus \{e\}}(u,v) \right]d\rho \\
      &+ \frac{1}{1-q} (q - p) o(d_G(o,x)).
   \end{align*}
   Again, the implicit constants in the little-o notation do not depend on $x, p,$ or $q$ (but may depend on $p_0$).
   %Recall that $M := (\log d_G(o,x))^C$ with some $C>2$ and so
   %\begin{align*}
   %   &\limsup_{x \to \infty} \frac{1}{d_G(o,x)} \left| O(M^{d+2})(q-p) + 
   %   O((Md_G(o,x))^{3d})\exp(-\sqrt{M})\int_p^q \frac{1}{1-\rho}d\rho \right| \\
   %   &\le
   %   \limsup_{x \to \infty} \frac{1}{d_G(o,x)} \left|
   %   O( (\log d_G(o,x))^{C(d+2)})(q-p)\right| \\
   %   &+ \limsup_{x \to \infty} \frac{1}{d_G(o,x)} \left| O( (\log d_G(o,x))^{3dC} d_G(o,x)^{3d} ) \exp(- (\log d_G(o,x))^{C/2} )\frac{1}{1-q}(q-p) 
   %   \right| \\
   %   =0.
   %\end{align*}
   Therefore we have
   \begin{align*}
      \limsup_{x \to \infty} \frac{\left| \E \tilde{D}_p(o,x) - \E \tilde{D}_q(o,x) \right|}{d_G(o,x)} 
      \le
      \limsup_{x \to \infty} \frac{1}{d_G(o,x ) p_0} \int_p^q 
      \E\left[ \sum_{e =\{u,v\} \in \bar{\pi}} \ind_{\left\{u \connectsthru{G_{\rho} \setminus \{e\}} v\right\}} 
      d_{G_{\rho} \setminus \{e\}}(u,v) \right]d\rho.
   \end{align*}
   To bound the right hand side, we employ Lemma \ref{lem:linear-bound}, taking
   \[
      X_e := \ind_{\left\{u \connectsthru{G_{\rho} \setminus \{e\}} v\right\}} 
      d_{G_{\rho} \setminus \{e\}}(u,v).
   \]
   Lemma \ref{lem:deleted-tails} then tells us that for some $c$,
   \[
      q_N := \Prob_p(X_e \le \exp(-cN))
   \]
   for all $\rho \in [p_0, 1]$, and therefore
   \[
      \sum_{N=0}^{\infty} N^{d+2} q_N \le \sum_{N=0}^{\infty} N^{d+2} \exp(-cN) =: B < \infty.
   \]
   Therefore by Lemma \ref{lem:linear-bound} we have $C=C(B,G)$ such that for any $L \in \N$
   \[
      \E\left[ \sum_{e \in \bar{\pi}} X_e \right] \le CL + \sum_{\ell \ge L} \ell \Prob(|\bar{\pi}| = \ell)^{1/2}
   \]
   and then by Proposition \ref{prop:pi-bar-bound}, whenever $d_G(o,x)$ is large enough,
   taking $L = K' d_G(o,x)$ then gives 
   \begin{align*}
      \E\left[ \sum_{e \in \bar{\pi}} X_e \right] &\le CK' d_G(o,x) + 
      \sum_{\ell \ge CK' d_G(o,x)} \ell \exp\left(-\frac{1}{2} \left(\frac{\ell}{K'}\right)^{1/5}\right) \\
      &\le (CK' + 1) d_G(o,x),
   \end{align*}
   where the second inequality holds when $d_G(o,x)$ is large enough.
   Thus, for $C, K'$ depending only on $G$ and $p_0$, for and $p,q \in [p_0, 1)$ we have
   \begin{align*}
      \limsup_{x \to \infty} \frac{\left| \E \tilde{D}_p(o,x) - \E \tilde{D}_q(o,x) \right|}{d_G(o,x)} 
      &\le
      \limsup_{x \to \infty} \frac{1}{d_G(o,x)p_0} \int_p^q  
      (CK' + 1) d_G(o,x) d\rho \\
      &= \frac{(CK + 1)}{p_0}|q-p|.
   \end{align*}
   That is, we have Lipschitz continuity for $\tilde{D}_p$ on $[p_0, 1)$ for any $p_0 > p_c$.
   To deduce the same for $D_p$, note that by Proposition \ref{prop:goodapprox} we have
   \begin{align*}
      \limsup_{x \to \infty} \frac{\left| \E D_p(o,x) - \E D_q(o,x) \right|}{d_G(o,x)} &\le
      \limsup_{x \to \infty} \frac{\left| \E \tilde{D}_p(o,x) - \E \tilde{D}_q(o,x) \right|}{d_G(o,x)} \\
      &+ \limsup_{x \to \infty} \frac{\E |D_p(o,x) - \tilde{D}_p(o,x) |}{d_G(o,x)}
      + \limsup_{x \to \infty} \frac{\E |D_q(o,x) - \tilde{D}_q(o,x) |}{d_G(o,x)} \\
      &\le (CK' + 1)|q - p|.
   \end{align*}
   To deduce Lipschitz continuity on the whole \emph{closed} interval $[p_0,1]$, it then
   only remains to show continuity at $p=1$.
   Note that $\E D_1(o,x) = D_1(o,x) = d_G(o,x)$ almost surely. So we want to show that
   \[
      \limsup_{p \to 1} \limsup_{x \to \infty} \frac{ \E[ D_p(o,x) - d_G(o,x) ] }{d_G(o,x)} = 0.
   \]
   Recalling that $D_p(o,x) = d_{G_p}(\mathring{o}, \mathring{x})$, by triangle inequality
   it suffices to show that
   \[
      \limsup_{p \to 1} \limsup_{x \to \infty} 
      \frac{ \E[ d_{G_p}(\mathring{o},\mathring{x}) - d_G(\mathring{o},\mathring{x}) ] }{d_G(o,x)} = 0
   \]
   and
   \[
      \limsup_{p \to 1} \limsup_{x \to \infty} 
      \frac{ \E[ |d_G(\mathring{o},\mathring{x}) - d_G(o,x)| ] }{d_G(o,x)} 
      \le
      \limsup_{p \to 1} \limsup_{x \to \infty} 
      \frac{ \E d_G(\mathring{o},o)  + \E d_G(x,\mathring{x}) }{d_G(o,x)}
      = 0.
   \]
   This latter equality follows from Proposition \ref{prop:smallholes}.
   
   To get the former equality, we first want to establish a bound
   on $ \E\left[\ind_{\{x' \connectsthru{G_p} y'\}} (d_{G_p}(x',y') - d_G(x',y') )^2 \right]$,
   using a very similar method to the proof of Theorem \ref{thm:main},
   but without coarse-graining; that is to say, using
   a very similar method to the proof of Theorem \ref{thm:finitelypresented}.
   
   Indeed, define given $x', y'$, define $F$ as in the proof of Theorem \ref{thm:finitelypresented}
   (union of $\Delta$-closed clusters of edges in the geodesic $[x',y']$.
   The same argument as in the proof of Theorem \ref{thm:finitelypresented}
   shows that if $x' \connectsthru{G_p} y'$ then
   \[
      0 \le d_{G_p}(x',y') - d_G(x,y) \le |N(F, \Delta)| \le D' |F|,
   \]
   where $D'$ is some constant depending on the graph $G$.
   
%   Say that two edges $e, e' \in E$ are \emph{$\Delta$-adjacent}
%   if there exists a cycle of diameter at most $\Delta$ containing both $e$ and $e'$.
%   If $e$ is a \emph{closed} edge, define its \emph{$\Delta$-closed component}
%   to be the minimal subset $C_{\Delta}(e)$ of $E$ containing $e$ such that
%   if $e' \in C_{\Delta}(e)$ and $e''$ is a closed edge which is $\Delta$-adjacent
%   to $e'$, then $e'' \in C_{\Delta}(e)$.
%   Given $p$ and $q$, fix an edge-geodesic $[x',y']$ from $p$ to $q$ in $G$,
%   and define $F$ to be the union of all the $\Delta$-closed components
%   of closed edges in $[x',y']$.
%   Even though Lemma \ref{lem:obstacles} is stated for forbidden
%   \emph{vertex} sets, the same proof shows that if $F \subset E$
%   is an \emph{edge} set, and $p$ is connected to $q$
%   in $G \setminus F$, then $p$ is connected to $q$ in 
%   $N(F, \Delta) \cup [x',y'] \setminus F$,
%   where here we can understand $N(F,\Delta)$
%   to be the set of edges which are $\Delta$-adjacent
%   to some edge of $F$.
%   Applying this to our choice of $F$,
%   and noting that by our construction of $F$
%   any edge of $(N(F,\Delta) \cup [x',y']) \setminus F$ is open,
%   we see that if $p \connectsthru{G_p} q$,
%   then there is an open path from $p$ to $q$
%   which lies in  $(N(F,\Delta) \cup [x',y'])$;
%   hence $d_{G_p}(x',y') \le |N(F,\Delta)| + d_G(x',y')$.
   
%   Since $G$ has bounded degree, there exists some
%   $D'$ such that each edge $e$ has at most $D'-1$
%   other edges which are $\Delta$-adjacent to it,
%   so $|N(F,\Delta)| \le D'|F|$.
   Once again $|F|$ is stochastically bounded by a sum of $d_G(x', y')$
   independent random variables $|\tilde{C}_i|$ (the sizes of the preclusters),
   where we have the bound
   \[
      \Prob_p( |\tilde{C}_i| = k ) \le [D'' (1 - p)]^k
   \]
   where $D''$ is some constant independent of $p$. So by Cauchy-Schwarz
   \begin{align*}
      \E\left[ \ind_{\left\{x' \connectsthru{G_p} y' \right\}} (d_{G_p}(x',y') - d_G(x',y'))^2 \right] 
      &\le
      D'^2 \E[ |F|^2 ] \\
      &\le D'^2 \E\left[ \left( \sum_{i=1}^{d_G(x',y')} |\tilde{C}_i| \right)^2\right] \\
      &\le D'^2 d_G(x',y')^2 \E[ |\tilde{C}|^2 ] \\
      &= d_G(x',y')^2 o_{p \to 1}(1).
   \end{align*}
   
%   Moreover, if we take $|\tilde{C}_{\Delta}(e)|$
%   to be independent random variables
%   with the same distribution as $|C_{\Delta}(e)|$,
%   we again have that $|F|$ is stochastically
%   dominated by $\sum_{e \in [x',y']} |\tilde{C}_{\Delta}(e)|$.
%   Thus we have
%   \begin{align*}
%      \E[ \ind_{p \connectsthru{G_p} q} (d_{G_p}(x',y') - d_G(x',y'))^2] 
%      &\le {D'}^2 \E\left[ \left( \sum_{e \in [x',y']} |\tilde{C}_{\Delta}(e)| \right)^2\right] \\
%      &\le {D'}^2 \sum_{e \in [x',y']} \E[ |\tilde{C}_{\Delta}(e)|^2 ]
%                                                            + {D'}^2 \left( \sum_{e \in [x',y'] } \E |\tilde{C}_{\Delta}(e)| \right)^2 
%   \end{align*}
%   Moreover a simple union bound gives
%   \[
%      \Prob(|\tilde{C}_{\Delta}(e)| = k) \le (2 D')^k (1-p)^k
%   \]
%   and hence
%   \begin{align*}
%      \E[ |\tilde{C}_{\Delta}(e)|] \le \sum_{k=1}^{\infty} k (2 D')^k (1-p)^k =: E_1(p), \\
%      \E[ |\tilde{C}_{\Delta}(e)|^2 ] \le \sum_{k=1}^{\infty} k^2 (2 D')^k (1-p)^k =: E_2(p).
%   \end{align*}
%   Thus we have 
%   \begin{align*}
%      \E[ \ind_{p \connectsthru{G_p} q} (d_{G_p}(x',y') - d_G(x',y'))^2] 
%      &\le {D'}^2( E_2(p) d_G(x',y') + E_1(p)^2 d_G(x',y')^2 )  \\
%      &\le {D'}^2 (E_2(p) + E_1(p)^2) d_G(x',y')^2.
%   \end{align*}
   So finally, by Cauchy-Schwarz and Proposition \ref{prop:smallholes}:
   \begin{align*}
      \E[ (d_{G_p}(\mathring{o},\mathring{x}) - d_G(\mathring{o},\mathring{x})) ]
      &= \sum_{x',y' \in V} \E[ \ind_{\mathring{o} = x', \mathring{x} = y'}  
      (d_{G_p}(\mathring{o},\mathring{x}) - d_G(\mathring{o},\mathring{x}))] \\
      &\le \sum_{x',y' \in V} \sqrt{ \Prob(\mathring{o}=x',\mathring{x}=y') 
                                                 \E\left[\ind_{\{x' \connectsthru{G_p} y'\}} (d_{G_p}(x',y') - d_G(x',y'))^2\right]} \\
      &\le o_{p \to 1}(1) \sum_{x',y' \in V} d_G(x',y') O\left(\exp(-\max(d_G(o,x'), d_G(x,y'))^{1/3})\right) \\
      &\le o_{p \to 1}(1) \sum_{k=0}^{\infty}O\left(k^{2d}\right) (d_G(o,x) + 2k) \exp(-k^{1/3}) \\
% 
 %     &\le D' \sqrt{E_2(p) + E_1(p)^2} d_G(o,x) \sum_{k=0}^{\infty} C(k+1)^{2d} (1 + 2k) \exp(-k^{1/3}) \\
 %     &= D'' \sqrt{E_2(p) + E_1(p)^2} d_G(o,x).
       &= o_{p \to 1}(1) d_G(o,x).
   \end{align*}
   Thus
   \begin{align*}
      \limsup_{p \to 1} \limsup_{x \to \infty} 
      \frac{|\E[ (d_{G_p}(\mathring{o},\mathring{x}) - d_G(\mathring{o},\mathring{x})) ]|}{d_G(o,x)}
      = 0,
   \end{align*}
   and we have continuity at $p=1$ and Lipschitz continuity on $[p_0, 1]$ for every $p_0 > p_c$.
   
   Lipschitz continuity of $\overline{d}$ and $\underline{d}$ then follow immediately from the fact that
   \[
      |\overline{d}(p) - \overline{d}(q)|, |\underline{d}(p) - \underline{d}(q)|
      \le \limsup_{x \to \infty} \frac{| \E[D_p(o,x)] - \E[D_q(o,x)] |}{d_G(o,x)}.
   \]
\end{proof}

\section{Future directions}
Theorem \ref{thm:finitelypresented} proves that an Antal-Pisztora type theorem
for \emph{any} transitive bounded degree coarsely simply connected graph
for $p$ sufficiently close to $1$ (and such that $p > p_u$, where $p_u$ again
is the \emph{uniqueness} threshold for $G$, see \cite{haggstrom2006uniqueness}).
It may seem obvious that this statement should hold for any (bounded degree transitive) graph,
but the proof relies crucially on coarse simple connectedness.
It is therefore natural to ask whether there exists a transitive bounded degree graph with $p_u < 1$
where this theorem fails, i.e. such that for any $p < 1$,
for any $K < \infty$,
$\Prob( x \connectsthru{G_p} y, d_{G_p}(x,y) \ge Kt )$
decays \emph{slower} than exponentially in $t$.
A tantalizing possibility is that the conclusion of Theorem \ref{thm:finitelypresented}
is \emph{equivalent} to being coarsely simply connected (among transitive graphs
with $p_u < 1$), so that \emph{every} such graph which is not
coarsely simply connected provides a counterexample. But this is far from clear.

Another natural question to ask is whether Theorem \ref{thm:finitelypresented} can be pushed
all the way down to $\tilde{p} = p_u$ for a larger class of graphs.
The proof here relies upon quantitative uniqueness bounds which are as yet unproven
in higher generality. However, a possibly more severe obstacle is
that the coarse-graining procedure outlined here is not as well-behaved geometrically
if the graph is not of polynomial growth; in particular, 
the degree of the coarse-grained graph can explode dramatically with the scale
(even though the coarse-grained graph remains $\Delta$-simply connected).
Thus a more subtle coarse-graining construction or a completely different approach may be required.

Another natural question is continuity of time constants for sufficiently large $p$
for general coarsely simply connected graphs (with $p_u < 1$).
Even though graphs which do not have polynomial growth are not expected to
have nice scaling limits, one can still ask whether Theorem \ref{thm:cty} holds as stated for an arbitrary graph.
Since we have good decay on the tails of bypasses $\ind_{\{x \connectsthru{G_p} y\}} d_{G_p}(x,y)$ 
in this regime, it seems very plausible that continuity of the time constants holds as well.
Besides missing results regarding distance to the infinite component and quantitative uniqueness which are used here,
and besides many crude union bounds which used polynomial volume growth,
polynomial growth was crucially used for the lattice animal bounds of Lemma \ref{lem:linear-bound}.
It is possible that more counting arguments can overcome this obstacle,
but a different approach may also be required.

\bibliographystyle{plain} 
\bibliography{chemicaldistancebib}

\begin{thebibliography}{10}

\bibitem{AP}
P.~Antal and A.~Pisztora.
\newblock {On the chemical distance for supercritical Bernoulli percolation}.
\newblock {\em The Annals of Probability}, 24(2):1036 -- 1048, 1996.

\bibitem{BenjaminiBabson}
E.~Babson and I.~Benjamini.
\newblock Cut sets and normed cohomology with applications to percolation.
\newblock {\em Proceedings of the American Mathematical Society},
  127(2):589--597, 1999.

\bibitem{BenjaminiTessera}
I.~Benjamini and R.~Tessera.
\newblock First passage percolation on nilpotent cayley graphs and beyond.
\newblock {\em Electronic Journal of Probability}, 20:1--20, 2015.

\bibitem{biskup2015isoperimetry}
M.~Biskup, O.~Louidor, E.~B. Procaccia, and R.~Rosenthal.
\newblock Isoperimetry in two-dimensional percolation.
\newblock {\em Communications on Pure and Applied Mathematics},
  68(9):1483--1531, 2015.

\bibitem{burton1989density}
R.~M. Burton and M.~Keane.
\newblock Density and uniqueness in percolation.
\newblock {\em Communications in mathematical physics}, 121:501--505, 1989.

\bibitem{CNN}
V.~H. Can, S.~Nakajima, and V.~Q. Nguyen.
\newblock Lipschitz-continuity of time constant in generalized first-passage
  percolation, 2024.

\bibitem{CantrellFurman}
M.~Cantrell and A.~Furman.
\newblock Asymptotic shapes for ergodic families of metrics on nilpotent
  groups.
\newblock {\em Groups Geom. Dyn.}, 11(4):1307--1345, 17.

\bibitem{CMT}
D.~Contreras, S.~Martineau, and V.~Tassion.
\newblock Supercritical percolation on graphs of polynomial growth, 2023.

\bibitem{de2024asymptotic}
L.~R. de~Lima.
\newblock Asymptotic shape of subadditive processes on groups and on random
  geometric graphs.
\newblock {\em arXiv preprint arXiv:2408.11615}, 2024.

\bibitem{GGTDrutuKapovich}
C.~Drutu and M.~Kapovich.
\newblock {\em Geometric Group Theory}.

\bibitem{garet2017continuity}
O.~Garet, R.~Marchand, E.~B. Procaccia, and M.~Th{\'e}ret.
\newblock Continuity of the time and isoperimetric constants in supercritical
  percolation.
\newblock {\em Electronic Journal of Probability}, 22:1, 2017.

\bibitem{Gorski}
C.~Gorski.
\newblock Strict monotonicity for first passage percolation on graphs of
  polynomial growth and quasi-trees, 2022.
\newblock To appear in The Annals of Probability.

\bibitem{grimmett1999percolation}
G.~Grimmett.
\newblock {\em Percolation}.
\newblock Springer, 1999.

\bibitem{haggstrom2006uniqueness}
O.~H{\"a}ggstr{\"o}m and J.~Jonasson.
\newblock Uniqueness and non-uniqueness in percolation theory.
\newblock {\em Probability Surveys}, 3:289--344, 2006.

\bibitem{Kingman}
J.~F.~C. Kingman.
\newblock The ergodic theory of subadditive stochastic processes.
\newblock {\em Journal of the Royal Statistical Society. Series B
  (Methodological)}, 30(3):499--510, 1968.

\bibitem{LSS}
T.~M. Liggett, R.~H. Schonmann, and A.~M. Stacey.
\newblock {Domination by product measures}.
\newblock {\em The Annals of Probability}, 25(1):71 -- 95, 1997.

\bibitem{DHS}
P.~Sosoe M.~Damron, J.~Hanson.
\newblock Sublinear variance in first-passage percolation for general
  distributions.
\newblock {\em Probability Theory and Related Fields}, 163(1):223--258, 2015.

\bibitem{Trofimov}
V.~I. Trofimov.
\newblock Graphs of polynomial growth.
\newblock {\em Math. USSR-Sb.}, 51:405--417, 1985.

\end{thebibliography}

\end{document}